%% file: connection_2017_05m_19d_final.tex
\documentclass[11pt]{ip-journal}
\usepackage[all]{xy}
\usepackage{color}
\usepackage{hyperref}
\usepackage{graphicx}

\newtheorem{Theorem}{Theorem}[section]
\newtheorem{Lemma}[Theorem]{Lemma}
\newtheorem{Corollary}[Theorem]{Corollary}

\newtheorem{Proposition}[Theorem]{Proposition}

\newtheorem{Remark}[Theorem]{Remark}

\def\a{\alpha}
\def \e{\epsilon}
\def \C{{\mathbb C}}

\newcommand{\pair}[1]{\left\langle #1 \right\rangle}
\newcommand{\norm}[1]{\left\|#1 \right\|}
\newcommand{\HOX}[1]{\marginpar{\footnotesize #1}}
\def\inter{\text{\rm int}}
\def\supp{\text{\rm supp}}
\def\diam{\text{\rm diam}}
\def\p{\partial}
\def\S{\mathcal S}
\def\R{\mathbb R}

\def\N{\mathbb N}
\def\Id{\mathop{Id}}

\def\End{\text{\rm End}}

\begin{document}
\title[Inverse problems for the connection Laplacian]{Inverse problems for the connection Laplacian}

\author[Y. Kurylev]{Yaroslav Kurylev}
\address{Department of Mathematics, University College London, Gower Street, London UK, WC1E 6BT.}
\email{y.kurylev@ucl.ac.uk}

\author[L. Oksanen]{Lauri Oksanen}
\address{Department of Mathematics, University College London, Gower Street, London UK, WC1E 6BT.}
\email{l.oksanen@ucl.ac.uk}

\author[G.P. Paternain]{Gabriel P. Paternain}
\address{ Department of Pure Mathematics and Mathematical Statistics,
University of Cambridge,
Cambridge CB3 0WB, UK}
\email {g.p.paternain@dpmms.cam.ac.uk}


\begin{abstract} We reconstruct a Riemannian manifold and a Hermitian vector bundle with compatible connection from the hyperbolic Dirichlet-to-Neumann operator associated with the wave equation of the connection Laplacian.
The boundary data is local and the reconstruction is up to the natural gauge transformations of the problem.
As a corollary we derive an elliptic analogue of the main result which solves a Calder\'on problem for connections on a cylinder.

\end{abstract}

\maketitle

\section{Introduction}

The purpose of the present paper is to show how to reconstruct a Riemannian metric and a Hermitian vector bundle with compatible connection from partial boundary measurements associated with the wave equation of the connection Laplacian (or rough Laplacian). The recovery is possible up to the natural gauges of the problem,
and the proof uses techniques from the Boundary Control method \cite{Belishev1987}.

There is considerable literature on the topic, and we shall review it in due course, but the strength of our results lies in the geometric generality involved:  there are no restrictions on the Riemannian manifold, Hermitian vector bundle or connection. Our methods also include a transparent and direct proof in the case of the trivial vector bundle that avoids gluing of local reconstructions. The problem is motivated by the Aharonov-Bohm effect which asserts that different gauge equivalence classes of electromagnetic potentials have different physical effects that can be detected by experiments. The solution to the inverse problem presented in this paper shows in great generality that different gauge equivalence classes of Hermitian connections (e.g. Yang-Mills
potentials) will have different boundary data and therefore are detectable by boundary measurements.

We proceed to state our results in more detail. Let $(M,g)$ be a smooth, compact, connected Riemannian manifold of dimension $m$ with non-empty boundary 
$\partial M$. 
Let $E\to M$ be a smooth Hermitian vector bundle of rank $n$,
and let us denote by $\pair{\cdot, \cdot}_E$ 
the Hermitian inner product on each fiber.
Let $\nabla$ be a connection compatible with the Hermitian structure, that is,
if we think of $\nabla$ as operating on sections
\[\nabla: C^\infty(M;E)\to C^\infty(M; E\otimes T^*M)\]
then for any pair $u,v \in C^\infty(M;E)$, we have 
\[d\langle u,v\rangle_{E}=\langle \nabla u, v\rangle_{E}+\langle u,  \nabla v \rangle_{E}.\]
Note that both the sides of the above equation are differential forms, that is, sections in $C^\infty(M; T^* M)$.

We can define a natural $L^{2}$-inner product of sections by setting
\[\langle u,v\rangle_{L^{2}(M;E)} =\int_{M}\langle u,v\rangle_{E}\,dx.\]
Here $dx$ is the Riemannian volume measure of $(M,g)$, and
we do not assume that $M$ is oriented.
Similarly we get a natural $L^{2}$-inner product in $C^\infty(M;E\otimes T^*M)$. The elements in
$C^\infty(M;E\otimes T^*M)$ can be thought of as 1-forms taking values in $E$. A pointwise product $\langle \alpha,\beta\rangle_{E}$ is a complex-valued 2-tensor on $M$ which can be contracted with $g$ to obtain a complex-valued function, and then integrated in $M$. In other words, if $\alpha=\alpha_{i}dx^{i}$ and $\beta=\beta_{i}dx^{i}$, then
\[\langle \alpha,\beta\rangle_{L^{2}(M;E\otimes T^*M)} =\int_{M}g^{ij}\langle \alpha_{i},\beta_{j}\rangle_{E}\,dx.\]
We denote by $\nabla^*$ the adjoint of $\nabla$ 
with respect to these $L^2$-inner products, and define the {\it connection Laplacian} as
\[P=\nabla^*\nabla.\]

We denote by $\End(E)$ the vector bundle whose fiber at $x \in M$ is the space of linear maps from the fiber $E_x$ to itself,
and say that a section $V \in C^\infty(M; \End(E))$ is a
{\it potential} if it is symmetric in the sense that for any pair of sections $u,v$ of $E$, 
\begin{align}
\label{symm_V}
\langle u,Vv\rangle_{E}=\langle Vu, v\rangle_{E}.
\end{align}

Let $V$ be a potential and consider the wave equation on sections, 
\begin{align}
&(\partial_{t}^{2}+P+V)u(t,x)=0,\;\;\;\;\;(0,\infty)\times M,\label{eq:wave}\\
&u|_{(0,\infty)\times\partial M}=f,\;\;\;\;\;\;\;\;\;\;(0,\infty)\times\partial M,\notag\\
&u|_{t=0}=\partial_{t}u|_{t=0}=0,\;\;\;\;\;\;\;\;\;\;\;\;\;\;\;\;\text{in}\;M.\notag
\end{align}
Let $T > 0$, let $\S \subset \p M$ be open, and 
define the restricted Dirichlet-to-Neumann operator
$$
\Lambda_\S^{2T}f = \nabla_{\nu}u|_{(0,2T) \times \S}, \quad f \in C_0^\infty((0,2T) \times \S; E),
$$
where $\nu$ is the interior unit normal on $\p M$
and $u$ is the solution of (\ref{eq:wave}).

Our main result is that, for a sharp time $T > 0$, 
the Hermitian vector bundle $E|_\S$ 
and the restricted Dirichlet-to-Neumann operator $\Lambda_\S^{2T}$
determine the Riemannian manifold $(M, g)$, the Hermitian vector bundle $E$,
the connection $\nabla$ and potential $V$.
Here $E|_\S$ is the pullback bundle $j^* E$
given by the inclusion map $j: \S \to M$.

\begin{Theorem}
Let $(M_{i},g_{i}, E_{i}, \nabla_{i}, V_i)$, $i=1,2$,
be two smooth Hermitian vector bundles that are defined on smooth, compact and connected Riemannian manifolds with boundary, and that are equipped with smooth Hermitian connections and smooth potentials. Suppose that $T > 0$ and open $\S_i \subset \p M_i$, $i=1,2$, satisfy
\begin{align*}
T > \max_{x \in M_i} d_{g_i}(x, \S_i), \quad i=1,2,
\end{align*}
where $d_{g_i}$ is the distance function on $(M_i, g_i)$.
Suppose, furthermore, that
there is a Hermitian vector bundle isomorphism $\phi:E_{1}|_{\S_1} \to E_{2}|_{\S_2}$ intertwining the Dirichlet-to-Neumann
operators
$\Lambda_{\S_{1}}^{2T}$ and $\Lambda_{\S_{2}}^{2T}$, that is, $\phi^*\Lambda_{\S_{2}}^{2T}= \Lambda_{\S_{1}}^{2T}\phi^{*}$. 
Then there is a Hermitian vector bundle isomorphism $\Phi:E_{1} \to E_{2}$ that covers an isometry between $(M_i, g_i)$, $i=1,2$,
and that satisfies $\Phi^*\nabla_{2}=\nabla_{1}$,
$\Phi^*V_{2}=V_{1}$
and $\Phi|_{E_{1}|_{\S_1}} = \phi$. 
\label{thm:maint}
\end{Theorem}

Let us denote by $\pi_i : E_i \to M_i$, $i=1,2$, the associated bundle projections, and recall that a vector bundle isomorphism $\Phi : E_1 \to E_2$ determines a diffeomorphism $\Psi : M_1 \to M_2$ via the equation $$\Psi \circ \pi_1 = \pi_2 \circ \Phi.$$ The isomorphism $\Phi$ covering an isometry means that $\Psi^* g_2 = g_1$.

It is a simple exercise to check that if an isomorphism 
$\Phi$ as in Theorem \ref{thm:maint} exists, 
then the restriction of $\Phi$ on $E_{1}|_{\S_1}$
intertwines the Dirichlet-to-Neumann operators. Hence Theorem \ref{thm:maint} is optimal in terms of the gauge invariances.

We recall that a generalized Laplacian $H$ on $E$
is a differential operator such that its principal symbol 
is 
$$
|\xi|^2 = g^{ij}(x) \xi_i \xi_j, 
\quad (x,\xi) \in T^* M,
$$
and we say that $H$ is symmetric if
$$
\pair{u, Hv}_{L^2(M;E)} = \pair{Hu, v}_{L^2(M;E)},
\quad u,v \in C_0^\infty(M;E).
$$
A symmetric generalized Laplacian $H$ on $E$
can be written in the form $P + V$ for some Hermitian connection $\nabla$
and potential $V$,
see e.g. \cite[Proposition 2.5]{Berline1992},
and wave equations for generalized Laplacians are the most 
general hyperbolic equations for which 
unique continuation is known to hold in the whole domain of influence,
see Theorem \ref{th_uniq_cont} below.
Such time sharp unique continuation, that goes back to the seminal paper \cite{Tataru1995},
is crucial to our proof.  

Let us also point out that  
if the symmetry assumptions in Theorem \ref{th_uniq_cont}
are weakened, then 
all the known uniqueness results in the scalar case 
require additional assumptions on the global geometry of $(M,g)$,
see \cite{Eskin2007,Kurylev2000a,Lassas2012}. 
We discuss 
the difficulties related to weaker symmetry assumptions
in more detail in Remark \ref{rem_nonsymm} below.

As a corollary of Theorem \ref{thm:maint}, let us consider the case when $(M,g)$ is known, 
$E$ is the trivial bundle $M\times\C^n$ with its usual Hermitian inner product and $V=0$.
Then $\nabla$ is of the form
\begin{align}
\label{connection_locally}
d_{A} = d+A,
\end{align}
where $A=A_{i}dx^{i}$
and each $A_{i}(x)$, $x \in M$, is a skew-Hermitian ($n\times n$)-matrix.
The Dirichlet-to-Neumann operator depends on $A$ and we write $\Lambda_{\p M}^{2T} = \Lambda_{\p M; A}^{2T}$. 

\begin{Corollary} 
Let $d_A$ and $d_B$ be two Hermitian connections on the trivial bundle $M \times \C^n$
over a fixed Riemannian manifold $(M,g)$,
and suppose that 
$\Lambda^{T}_{\p M; A}=\Lambda^{T}_{\p M; B}$ for
$T > \max_{x \in M} d_g(x, \p M)$.
Then, there exists a smooth $U:M\to U(n)$ such that $U|_{\partial M}=\Id$ and 
\begin{align}
\label{AB_gauge}
B=U^{-1}dU+U^{-1}AU.
\end{align}
\label{cor:test}
\end{Corollary}

Note that if $A$ and $B$ satisfy (\ref{AB_gauge}),
then $U^{-1}d_{A}U=d_{B}$ and hence
$P_{B}=U^{-1}P_{A}U$, where $P_i = d_i^* d_i$, $i=A,B$. Thus if $u$ solves the wave equation for $P_{B}$, then $Uu$ solves it
for $P_{A}$. Hence the above corollary can not be improved, that is, 
 if  $U:M\to U(n)$ satisfies $U|_{\partial M}=\Id$
and (\ref{AB_gauge}) holds, then $\Lambda^{T}_{\p M; A}=\Lambda^{T}_{\p M; B}$ for any
$T$. In the context of the gauges in Theorem \ref{thm:maint}, we have
that $\phi$ is the identity and $\Phi(x,s)=(x,U(x)s)$, where $(x,s)\in M\times\C^n$.

The situation of the corollary is the one that appears in the literature. For the abelian case $n=1$, the corollary in essentially proved in \cite{Kurylev1995} via the Boundary Control method.
The Boundary Control method
was pioneered for the isotropic wave equation on a domain in \cite{Belishev1987} and developed for manifolds in \cite{Belishev1992}.
Note, however, that in \cite{Kurylev1995} the boundary spectral data is used, and therefore the result 
does not give the sharp time $T$. 

In \cite{Finch2001}, the corollary is proved under the further assumptions that $M$ is 
a two dimensional domain,  $g$ is the Euclidean metric tensor and the connection is small in a suitable sense.
The proof uses geometric optics solutions and reduces the problem to an injectivity result about the non-abelian Radon transform, which is of independent interest; see \cite{Eskin2004} for the case of the Euclidean metric
and compactly supported connections.  More recently, the injectivity result for the non-abelian Radon transform was extended to any simply connected surface with strictly convex boundary and no conjugate points \cite{Paternain2012} and to higher dimensions and negative curvature \cite{Guillarmou2015}.

There is a result due to G. Eskin \cite{Eskin2005}  that implies Corollary \ref{cor:test} under the assumption that $M$ is a domain in Euclidean space with obstacles. 
Our proof seems however simpler. Eskin also proves a related theorem for the case of time-dependent Yang-Mills potentials in \cite{Eskin2008a}. A survey on these results, including amended statements, is given in \cite{Eskin2006}. 

The proof of Corollary \ref{cor:test} follows directly from of our local reconstruction procedure and well-known properties of the cut locus, so the full power of Theorem \ref{thm:maint} is not needed. As far as we are aware, there are no previous results for this problem when the bundle is not trivial; perhaps the closest in spirit is the result in \cite{Kurylev2009a} for the hyperbolic Dirac equation. However in this reference it is assumed that the data is given on
the whole boundary for an infinite time interval, whereas our main result assumes only partial data and is sharp in terms of $T$. One of the main contributions of the present paper is to develop a  new method to glue local reconstructions. The method allows us to reconstruct an isomorphic copy of the structure $(g,E,\nabla,V)$ on the interior of $M$ given the data $\Lambda_\S^{2T}$ corresponding to a sharp time $T$.

Let us mention that there is a recent stability result for  Gel'fand's inverse interior spectral problem \cite{Bosi2017}.
There the problem is studied for compact Riemannian manifolds without boundary, and the result is closer to Corollary \ref{cor:test} than Theorem \ref{thm:maint} in the sense that the global geometry needs to be considered only along the cut locus of suitable semi-geodesic coordinates. 
As the proof in \cite{Bosi2017} uses also techniques from the Boundary Control method, we conjecture that if $d_A$ and $d_B$ are as in Corollary \ref{cor:test}, then 
$
d_1(\mathcal O(B), \mathcal O(A)) \le \omega(d_2(\Lambda^{T}_{\p M; A}, \Lambda^{T}_{\p M; B}))
$
where $\omega$ is a modulus continuity of the same double logarithmic type as in \cite{Bosi2017},
$d_1$ and $d_2$ are suitable distance functions,
and 
$$
\mathcal O(A) = \{U^{-1} A U + U^{-1} d U;\ U : M \to U(n),\ U|_{\p M} = Id \}
$$
is the orbit of $A$ under the gauge group.

As a final corollary,
let us consider an elliptic analogue of Theorem \ref{thm:maint}.
This application is very much in the spirit of \cite[Theorem 1.5]{Ferreira2013}
where an elliptic scalar valued equation was considered.

Let $(M_0,g_0)$ be a compact, connected  Riemannian manifold with boundary, and let $C=\mathbb{R}\times M_0$ be the infinite cylinder with the product metric $g=dt^2+g_0$.
Here $dt^2$ is the Euclidean metric on $\R$.
We consider a Hermitian vector bundle $E_0 \to M_0$ with a Hermitian connection $\nabla_{0}$, and
define the operator $P_{0}=\nabla^*_{0}\nabla_{0}$.
Moreover, we have an induced Hermitian bundle $E$ with connection $\nabla$ on $C$, that is,
 $E=\pi^*E_0$ and $\nabla=\pi^*\nabla_0$, where $\pi:C\to M_0$ is the canonical projection.

Let us denote by 
$\lambda_{1}\leq \lambda_{2}\leq \dots$ the Dirichlet eigenvalues of the operator $P_0$.
A point $\lambda\in \mathbb{C}\setminus[\lambda_{1}, \infty)$ is not in the continuous spectrum of the operator $\nabla^*\nabla=-\partial_{t}^{2}+P_0$ and, for any $f\in C_{0}^{\infty}(\partial C;E)$, the equation
\[(-\partial_{t}^{2}+P_0-\lambda)u=0\;\text{in}\;C,\;\;\;\;\;u|_{\partial C}=f,\]
has a unique bounded solution $u\in C^{\infty}(C;E)$.
We define the elliptic Dirichlet-to-Neumann map
\[
\Lambda(\lambda) f = \nabla_{\nu} u|_{\p C}, \quad
\Lambda(\lambda):C_{0}^{\infty}(\partial C;E)\to C^{\infty}(\partial C;E).\]

Our application is the following recovery result:

\begin{Corollary} 
The Hermitian vector bundle $E|_{\p C}$ and the elliptic Dirichlet-to-Neumann map
$\Lambda(\lambda)$ for a fixed $\lambda\in \mathbb{C}\setminus[\lambda_{1},\infty)$ determine the structure $(M_{0},g_{0},E_{0},\nabla_{0})$.
\label{thm:ell}
\end{Corollary}

Here, the structure is determined up to the natural gauge invariances as in Theorem \ref{thm:maint}.
It is possible to prove also a version of the corollary assuming that $\lambda$ is in the continuous spectrum of $-\partial_{t}^{2}+P_0$ as long as it avoids the eigenvalues $\lambda_i$. This extension can be carried out as
in \cite[Theorem 1.7]{Ferreira2013} but we do not include it here.

This paper is organized as follows. Section 1 is the introduction and states the main results. In Section 2 we include preliminaries, mostly having to do with the direct problem, finite speed of propagation, unique continuation and approximate controllability. The results here are standard, but some details are provided to ensure the usual techniques fit our setting.  Section 3 contains the local reconstruction procedure near the boundary.  We first reconstruct the metric $g$ and the core of the section is the reconstruction of the Hermitian bundle and the connection. The main local result is Theorem \ref{th_local_connection} and Corollary \ref{cor:test} is immediately derived from this theorem and well-known properties of the cut locus. Section 4 contains the global reconstruction procedure, expains in detail how to build up the structure from local data and finishes the proof of Theorem \ref{thm:maint}. In the final Section 5 we prove Corollary \ref{thm:ell}.

\section{Preliminaries}

\subsection{Local trivializations}

The connection $\nabla$ is of the form (\ref{connection_locally})
on a local trivialization of $E$.
Let us derive local expressions for $d_{A}^*$ and $P = d_{A}^*d_{A}$.
To this end, we consider a section $u: M \to E$ and
a $E$-valued 1-form $\beta=\beta_{i}dx^{i}$
supported on a local trivialization. As $A$ is skew-hermitian,
\begin{align*}
\langle Au,\beta\rangle_{L^{2}(M;E\otimes T^*M)}= \int_{M}g^{ij}\langle A_{i}u,\beta_{j}\rangle_{E}\,dx=-\int_{M}\langle u,g^{ij}A_{i}\beta_{j}\rangle_{E}\,dx.
\end{align*}
We define
$(A,\beta) = g^{ij}A_{i}\beta_{j}$
and see that
$d_{A}^{*}=d^{*}-(A,\cdot)$.
Thus
\[Pu=d^{*}du+d^*(Au)-(A,du)-(A,Au).\]
We recall that for a 1-form $\alpha$ in local coordinates
$$
d^* \alpha = - |g|^{-1/2} \frac{\p}{\p x^i} \left( |g|^{1/2} g^{ij} \alpha_j \right),
$$
hence $d^*(Au)=(d^{*}A)u-(A,du)$, and
\begin{equation}
Pu=d^*du-2(A,du)+(d^*A)u-(A,Au).
\label{eq:P}
\end{equation}
This exposes the nature of $P$: the principal part is the usual Laplacian and the first order term given by $-2(A,du)$. 


When working near the boundary $\p M$, it is convenient to use 
boundary normal coordinates, that is, 
semigeodesic coordinates adapted to the boundary.
Let $\Gamma \subset \p M$ be open. Then 
the semigeodesic coordinates adapted to $\Gamma$ are 
given by the map 
\begin{align}
\label{semigeodesic_coord}
(s,y) \mapsto \gamma(s; y, \nu), \quad y \in \Gamma,\ s \in [0, \sigma_\Gamma(y)),
\end{align}
where the cut distance $\sigma_\Gamma : \Gamma \to (0,\infty)$ is defined by 
\begin{align} 
\label{def_cut_distance}
\sigma_\Gamma(y) &= \max \{ s \in (0, \tau_M(y)];\ d_g(\gamma(s; y, \nu), \Gamma) = s\},
\\\notag
\tau_M(y) &= \sup \{ s \in (0, \infty);\ \gamma(s; y, \nu) \in M^\inter \}.
\end{align}
Here 
$\gamma(\cdot; x, \xi)$ is the geodesic with the initial data $(x,\xi) \in TM$.
We recall that $\nu$ is the interior unit normal on $\p M$,
and define 
\begin{align}
\label{semigeod_strip}
M_\Gamma = \{\gamma(s;y,  \nu);\ y \in \Gamma,\ s \in [0, \sigma_\Gamma(y))\}. 
\end{align}
Then a point $x \in M_\Gamma$ is represented in the coordinates (\ref{semigeodesic_coord}) by $(s, y)$, 
where $s$ is the distance $d_g(x,\Gamma)$ and $y$ is the unique closest point to $x$ in $\Gamma$.
Moreover, $g$ has the form $ds^2 + h_{jk}(s,y) dy^j dy^k$ and the principal part of $P$ is 
\begin{align}
\label{eq_P_in_boundary_normal}
-\p_s^2 - h^{jk}(s,y) \p_{y^j} \p_{y^k}.
\end{align}

\subsection{The direct problem}
\label{sec_direct}

Let us consider the initial-boundary value problem
\begin{align}
\label{eq_wave_initialbd}
&(\partial_{t}^{2}+P + V)u(t,x)=F, \quad (0,T)\times M,\\
&u|_{(0,T)\times\partial M}=f, \quad (0,T)\times\partial M,\notag\\
&u|_{t=0}=\psi,\ \partial_{t}u|_{t=0}=\phi, \quad \text{in}\;M,\notag
\end{align}
where $T > 0$.
When $f = 0$ we have the energy estimate
\begin{align}
\label{energy_estimate}
&\norm{u(t)}_{H_0^1(M;E)} + \norm{\p_t u(t)}_{L^2(M;E)}
\\\notag&\qquad\le C (\norm{\psi}_{H_0^1(M;E)} + \norm{\phi}_{L^2(M;E)} + \norm{F}_{L^2((0,t) \times M;E)}),
\end{align}
for all $t \in (0,T)$.
For a proof in the scalar valued case, we refer to \cite[Section 7.2]{Evans1998}. The proof is analogous in the vector valued case and we omit it. 
We have also higher regularity results under suitable compatibility conditions. In what follows, we need only the following estimate 
\begin{align}
\label{energy_estimate_higher}
&\norm{u}_{H^m((0,T) \times M;E)} \le C (\norm{\phi}_{H^{m-1}(M;E)} + \norm{F}_{H^{m-1}((0,T) \times M;E)}),
\end{align}
where $m \ge 1$, $f$ and $\psi$ vanish, $F$ is compactly supported in the time interval $(0,T)$ (but not necessarily in space), and $\phi$ 
is compactly supported in $M^\inter$, see e.g. \cite{Evans1998}. 
We can extend $f \in C_0^\infty((0,\infty) \times \p M; E)$
as a smooth function on the whole domain $(0,\infty) \times M$ and substract it from $u$. By using 
(\ref{energy_estimate_higher}) we see that the solution of (\ref{eq:wave}) is smooth for such sources $f$.

We need a sharp regularity result for the Neumann trace.
The result is due to Lasiecka, Lions and Triggiani
in the scalar valued case \cite{Lasiecka1986}.
The proof in the present setting is analogous but we give it for the convenience of the reader. 
We will use the following identity
\begin{align}
\label{intpart_nabla}
&\langle \nabla^*u,v\rangle_{L^{2}(M;E)} - \langle u, \nabla v\rangle_{L^{2}(M;E\otimes T^*M)} 
\\\notag&\qquad= 
\langle d^*u,v\rangle_{L^{2}(M;E)} - \langle u,dv\rangle_{L^{2}(M;E\otimes T^*M)} 
= \int_{\partial M}\langle i_{\nu}u,v\rangle_{E}\,dS,
\end{align}
where $u \in C^\infty(M;E\otimes T^*M)$, $v \in C^\infty(M;E)$,
and $dS$ is the Riemannian volume of $(\p M, g)$.
This follows from \cite[Prop. 2.9.1]{Taylor1996} since the principal symbol of $\nabla$ coincides with the principal symbol of $d$.

\begin{Theorem}
\label{th_regularity_neumann}
Suppose that $F$, $f$ and $\psi$ vanish and let $\phi \in L^2(M;E)$.
Then the solution $u$ of (\ref{eq_wave_initialbd}) satisfies 
$\nabla_\nu u \in L^2((0,T)\times\partial M; E)$.
\end{Theorem}
\begin{proof}
We will first suppose that $\phi \in C_0^\infty(M;E)$.
Then $u$ is smooth by (\ref{energy_estimate_higher}).
We extend $\nu$ as a smooth vector field on the whole domain $M$,
and denote this extension still by $\nu$.
We have 
\begin{align*}
&\pair{Pu, \nabla_\nu u}_{L^2((0,T) \times M;E)}
\\&\quad= \pair{\nabla u, \nabla \nabla_\nu u}_{L^2((0,T) \times M; E\otimes T^*M)}
+ \int_0^T \int_{\partial M}|\nabla_\nu u|_E^2 \,dS.
\end{align*}
Here $|u|_E^2 = \pair{u,u}_E$.
In local coordinates, the principal part of both 
$$
\pair{\nabla_j u, \nabla_k \nabla_\nu u}_E g^{jk} \quad \text{and} \quad 
\frac 1 2 \nu (\pair{\nabla_j u, \nabla_k u}_E g^{jk})
$$ 
is $\pair{\p_j u, \nu^p \p_p \p_k u}_E g^{jk}$. Thus 
\begin{align*}
\pair{\nabla u, \nabla \nabla_\nu u}_{L^2((0,T) \times M; E\otimes T^*M)}
= \frac 1 2 \int_0^T \int_M \nu (\pair{\nabla_j u, \nabla_k u}_E g^{jk}) dx + R,
\end{align*}
where the remainder term $R$ satisfies $|R| \le C \norm{u}_{H^1((0,T) \times M;E)}^2$.
Moreover, 
\begin{align*}
\int_0^T \int_{M} \nu (\pair{\nabla_j u, \nabla_k u}_E g^{jk}) dx 
&= - \int_0^T \int_{M} (\mathop{\rm div} \nu) \pair{\nabla_j u, \nabla_k u}_E g^{jk} dx 
\\&\qquad- \int_0^T \int_{\p M} \pair{\nabla_j u, \nabla_k u}_E g^{jk} dS.
\end{align*}

As $u$ vanishes on the boundary, we have in the boundary normal coordinates $(s, y) \in [0, \epsilon) \times \p M$ that
$$
\pair{\nabla_j u, \nabla_k u}_E g^{jk} = |\p_s u|_E^2 = |\nabla_\nu u|_E^2.
$$
Hence
\begin{align}
\label{eq_LLT_Pu_term}
\pair{Pu, \nabla_\nu u}_{L^2((0,T) \times M; E)}
= \frac 1 2 \norm{\nabla_\nu u}_{L^2((0,T) \times \p M;E)}^2 + R,	
\end{align}
where the remainder term $R$ satisfies $|R| \le C \norm{u}_{H^1((0,T) \times M;E)}^2$.

Analogously 
\begin{align*}
&\pair{\p_t^2 u, \nabla_\nu u}_{L^2((0,T) \times M; E)}
\\&\qquad= -\frac 1 2 \int_0^T \int_M \nu \pair{\p_t u, \p_t u}_E dx
+ \left[ \int_{M} \pair{\p_t u, \nabla_\nu u}_E dx \right]_{t=0}^{t=T},
\end{align*}
and 
\begin{align*}
\int_0^T \int_M \nu \pair{\p_t u, \p_t u}_E dx
&= - \int_0^T \int_M (\mathop{\rm div} \nu) \pair{\p_t u, \p_t u}_E dx
\\&\qquad- \int_0^T \int_{\p M} \pair{\p_t u, \p_t u}_E dx,
\end{align*}
where the second term on the right-hand side is zero since $u=0$ on $\p M$.
Hence
\begin{align}
\label{eq_LLT_ptu_term}
&|\pair{\p_t^2 u, \nabla_\nu u}_{L^2((0,T) \times M;E)}|
\le C \norm{u}_{H^1((0,T) \times M;E))}^2 
\\\nonumber&\qquad\qquad\qquad
+ C \max_{t=0,T}(\norm{u(t)}_{H_0^1(M;E)}^2 + \norm{\p_t u(t)}_{L^2(M;E)}^2).
\end{align}
Clearly
\begin{align}
\label{eq_LLT_Vu_term}
|\pair{V u, \nabla_\nu u}_{L^2((0,T) \times M;E)}|
\le C \norm{u}_{H^1((0,T) \times M;E)}^2.
\end{align}
Combining (\ref{eq_LLT_Pu_term})-(\ref{eq_LLT_Vu_term}) with the energy estimate (\ref{energy_estimate}),
we get
$$
\norm{\nabla_\nu u}_{L^2((0,T) \times \p M;E)}^2 \le C \norm{\phi}_{L^2(M;E)}^2.
$$
The claim follows since $C_0^\infty(M;E)$ is dense in $L^2(M;E)$.
\end{proof}

We will next discuss how (\ref{eq_wave_initialbd}) 
can be solved for non-smooth $\phi$ that are supported in the interior of $M$.
Let $K \subset M^\inter$ be compact and choose $\chi \in C_0^\infty(M)$ such that $\chi = 1$ near $K$.
Define first the map
$$
\mathcal W : H_0^{m-1}((0,T) \times M;E) \to 
H_0^{m}(M;E), 
\quad \mathcal WF = \chi u(T),
$$
where $u$ solves (\ref{eq_wave_initialbd}) 
with $f = 0$, $\psi = 0$ and $\phi = 0$.
For $\phi \in C_0^\infty(M; E)$ 
satisfying $\supp(\phi) \subset K$
it holds that 
$$
\pair{\mathcal W F, \phi}_{L^2(M; E)} = 
\pair{F, v}_{L^2((0,T) \times M; E)} 
$$
where $v$ is the solution of 
\begin{align}
\label{adjoint_of_W}
&(\partial_{t}^{2}+P+V)v(t,x)=0, \quad (0,T)\times M,\\
&v|_{(0,T)\times\partial M}=0, \quad (0,T)\times\partial M,\notag\\
&v|_{t=T}=0,\ \partial_{t}v|_{t=T}=-\phi, \quad \text{in}\;M.\notag
\end{align}
Then the adjoint of $\mathcal W$,
restricted on the subspace 
$$\dot H^{-m}(K;E) = \{\phi \in H^{-m}(M;E);\ \supp(\phi) \subset K\},$$
is the unique continuous extension 
$$
\dot H^{-m}(K;E) \to H^{-m+1}((0,T) \times M;E)
$$
of the map
solving (\ref{adjoint_of_W}) for smooth $\phi$.
We may reverse time to get the solution $u$ of 
(\ref{eq_wave_initialbd}) with $\phi \in \dot H^{-m}(K;E)$
and $F$, $f$ and $\psi$ vanishing.

Let us now consider the traces of such a solution $u$.
As the principal part of $P+V$ is of the form (\ref{eq_P_in_boundary_normal})
in the boundary normal coordinates $(s,y) \in [0, \epsilon) \times \p M$,
we may repeat the proof of \cite[Th. B.2.9]{Hormander1985} without any changes in the present, vector valued setting. 
This implies that
$u$ is in $\overline H_{(\mu,\sigma)}^{loc}((0,T) \times (0, \epsilon) \times \Gamma; E)$ where $\mu + \sigma \le -m + 1$ and $\Gamma \subset \p M$ is a coordinate neighbourhood.
Taking now large $\mu \in \R$ and small $\sigma \in \R$,
we may apply \cite[Th. B.2.7]{Hormander1985}
to see that there is $m' \in \R$ such that 
that the maps $s \mapsto u(\cdot, s, \cdot)$ and $s \mapsto \nabla_\nu u(\cdot, s, \cdot)$
are continuous with values in $H^{-m'}((0,T) \times \Gamma;E)$.
In particular, the traces $u|_{(0,T) \times \p M}$ and $\nabla_\nu u|_{(0,T) \times \p M}$ are well-defined for 
the solution $u$ of 
the wave equation (\ref{eq_wave_initialbd}) with $\phi \in \dot H^{-m}(K;E)$
and $F$, $f$ and $\psi$ vanishing.

\subsection{Finite speed of propagation, unique continuation and approximate controllability}

The equation (\ref{eq:wave}) has the following
finite speed of propagation property:

\begin{Theorem}
\label{th_finite_speed}
Let $T > 0$, $U \subset M$ be open and define the cone
$$
\mathcal C = \{(t,x) \in (0,T) \times M;\ d_g(x, U) < T- t \}.
$$
Suppose that $f \in C_0^\infty((0,T) \times \p M;E)$
vanishes in the intersection 
$$
\mathcal C \cap ((0,T) \times \p M).
$$
Then the solution $u$ of (\ref{eq:wave}) vanishes in $\mathcal C$.
In particular, if $\Gamma \subset \p M$ is open, $r \in (0,T)$,
and $\supp(f) \subset (T-r,T) \times \Gamma$, then 
$
\supp(u(T))$ is contained in the domain of influence 
$$
M(\Gamma, r) = \{x \in M;\ d_g(x,\Gamma) \le r\}.
$$
\end{Theorem}

We refer to \cite[Lemma 4.1]{Isozaki2014}
for a proof in the scalar valued case. 
The proof in the present setting is analogous and we omit it.

The operator $P+V$ is of principally scalar form, and the local unique continuation result \cite{Eller2002} can be applied. 
The local result implies the following result due to Eller and Toundykov \cite{Eller2012} that is analogous to the semi-global Holmgren theorem.

\begin{Theorem}
\label{th_uniq_cont}
Let $T>0$ and let $\Gamma \subset \p M$ be open.
Let $s \in \R$, and suppose that $u \in H^{s}((0,2T) \times M;E)$
satisfies $(\p_t^2 + P+V) u = 0$ and
$$
u|_{(0,2T) \times \Gamma} = 0, \quad \nabla_\nu u|_{(0,2T) \times \Gamma} =0.
$$
Then $u(T,x) = 0$ whenever $x \in M(\Gamma,T)^\inter$.
\end{Theorem}

Let us denote $Wf = u(T)$, where $u$ is the solution of (\ref{eq:wave}).
The formal adjoint of $W$ is $W^* \phi = \nabla_\nu v|_{(0,T) \times \p M}$,
where $v$ is the solution of (\ref{adjoint_of_W}). 
Indeed,
\begin{align*}
0 &= \pair{(\partial_{t}^{2}+P+V)u, v}_{L^2((0,T) \times M;E)} -
\pair{u, (\partial_{t}^{2}+P+V)v}_{L^2((0,T) \times M;E)}
\\ \nonumber
&= 
\left[ \pair{\p_t u, v}_{L^2(M;E)} - \pair{u, \p_t v}_{L^2(M;E)}\right]_{t=0}^{t=T}
\\ \nonumber
&\qquad+
\pair{\nabla_\nu u, v}_{L^2((0,T) \times \p M;E)} - \pair{u, \nabla_\nu v}_{L^2((0,T) \times \p M;E)}
\\ \nonumber
&= 
\pair{u(T), \phi}_{L^2(M;E)}
- \pair{f, \nabla_\nu v}_{L^2((0,T) \times \p M;E)}.
\end{align*}
As discussed in the end of the previous section,
for any $m \in \R$ and compact $K \subset M^\inter$
there is $m' \in \R$ such that 
\begin{align}
\label{Wstar_continuity}
W^* : \dot H^{-m}(K;E) \to H^{-m'}((0,T) \times \p M;E).
\end{align}
For the purposes of the present paper,
apart from the case $m=0$ described in Theorem \ref{th_regularity_neumann},
the optimal value of $m'$ is irrelevant.
We may consider $L^2(K; E)$ 
as the subspace of $L^2(M; E)$ consisting of functions supported in $K$ (i.e. $\dot H^0(K;E)$).

If $\Gamma \subset \p M$ is open and nonempty and $r > 0$, then
the map 
$$
\phi \mapsto \nabla_\nu v|_{(0,r) \times \Gamma} : 
L^2(M(\Gamma, r);E) \to L^2((0,r) \times \Gamma; E)
$$ 
is injective by Theorem \ref{th_uniq_cont}. A duality argument implies that the wave equation (\ref{eq:wave})
is approximately controllable in the sense of the lemma below. 
This is well-known in the scalar valued case, see e.g. \cite{Katchalov2001}.
The proof in the present setting is analogous, however, we give it for the convenience of the reader. 

\begin{Lemma}
\label{lem_density}
Let $\Gamma \subset \p M$ be open and $r > 0$. Then
\begin{align}
\label{set_uT}
\{Wf ;\ f \in C_0^\infty((T-r, T) \times \Gamma; E) \}
\end{align}
is dense in $L^2(M(\Gamma,r);E)$.
\end{Lemma}
\begin{proof}
By the finite speed of propagation,
the set (\ref{set_uT}) is a subspace of $L^2(M(\Gamma,r);E)$.
It is enough to show that the orthogonal complement of this subspace contains only the origin. 
Suppose that $\phi \in L^2(M(\Gamma,r);E)$ satisfies 
\begin{align}
\label{ortho_lem_dens}
(Wf,\phi)_{L^2(M;E)}=0, \quad 
f \in  C_0^\infty((T-r, T) \times \Gamma; E). 
\end{align}
Recall that $W^* \phi = \nabla_\nu v|_{(0,T) \times \p M}$
where $v$ is the solution of (\ref{adjoint_of_W}). 
Hence (\ref{ortho_lem_dens}) implies that $\nabla_\nu v|_{(T-r,T) \times \Gamma} = 0$. 
We extend $v$ across the surface $t=T$ by using the odd reflection $v(t,x)=-v(2T-t,x)$. Then the extension satisfies the wave equation 
\begin{align*}
&(\partial_{t}^{2}+P+V) v(t,x)=0, \quad (0,2T)\times M,\\
&v|_{(0,2T)\times\partial M}=0, \quad (0,2T)\times\partial M,\notag\\
&v|_{t=T}=0,\ \partial_{t} v|_{t=T}=-\phi, \quad \text{in}\;M,\notag
\end{align*}
together with the additional boundary condition 
$\nabla_\nu v|_{(T-r,T+r) \times \Gamma} = 0$.
Theorem \ref{th_uniq_cont} implies that $\phi = 0$.
Here we used also the fact
that the boundary of $M(\Gamma, r)$ is of measure zero \cite{Oksanen2011}.
\end{proof}

As described in the scalar valued case in Section 4.4 of \cite{Lassas2012},
in order to determine the cut distance $\sigma_\Gamma$ from the restricted Dirichlet-to-Neumann map, we need to use a perturbation argument that is based on a refined version of approximate controllability and modified domains of influence. 
Let $\Gamma \subset \p M$ and $h : \Gamma \to \R$, and define
\begin{align*}
M(\Gamma, h) & = \{x \in M;\ 
\inf_{y \in \Gamma} (d_g(x, y) - h(y)) \le 0\},
\end{align*}
and denote for $T > 0$
\def\B{\mathcal B}
\begin{align*}
\B(\Gamma, h; T) = \{(t, y) \in (0, T) \times \Gamma;\ T - h(y) < t \}. 
\end{align*}
If $r > 0$ and $h(y) = r$, $y \in \Gamma$,
then $M(\Gamma, h)$ coincides with our earlier definition of $M(\Gamma, r)$.
We denote by $1_S$ the indicator function of a set $S \subset M$,
that is, $1_S(x) = 1$ if $x \in S$ and $1_S(x) = 0$ otherwise.

For the convenience of the reader, we give a proof of the following lemma. An analogous lemma is stated in \cite{Lassas2012}
without a proof. 

\begin{Lemma}
\label{lem_density_h}
Let $T > 0$ and suppose that
$\Gamma \subset \p M$ is open.
Let $L \in \N$, let $\Gamma_\ell \subset \Gamma$ be open
and let $h_\ell \in C(\bar \Gamma_\ell)$, $\ell = 1, \dots, L$.
We define 
\begin{align}
\label{def_h}
h = \sum_{\ell = 1}^L h_\ell 1_{\Gamma_\ell},
\end{align}
and suppose that $h \le T$ pointwise.
Then 
\begin{align}
\label{density_h_Wf}
\{ Wf;\ f \in C_0^\infty(\B(\Gamma, h; T); E) \}
\end{align}
is dense in $L^2(M(\Gamma, h); E)$.
\end{Lemma}
\begin{proof}
Let $\epsilon > 0$.
There is a simple function 
\begin{align*}
h_\epsilon(y) = \sum_{j=1}^J T_j 1_{\Gamma_j}(y),
\end{align*}
where $J \in \N$, $T_j \in (0,T)$ and $\Gamma_j \subset \Gamma$ are open and disjoint, 
such that $h < h_\epsilon + \epsilon$ almost everywhere on $\Gamma$ and 
$h_\epsilon < h$ on $\bar \Gamma$, see e.g. \cite[Lemma 4.2]{Oksanen2011}.

We show by induction on $J$ that the density holds when $h = h_\epsilon$.
The base case $J=1$ follows from Lemma \ref{lem_density}.
We define $\tilde h_\epsilon = h_\epsilon - T_J 1_{\Gamma_J}$,
and use the shorthand notation $M_0 = M(\Gamma, \tilde h_\epsilon)$
and $M_1 = M(\Gamma_J, T_J)$.
Let $\psi \in L^2(M(\Gamma, h_\epsilon); E)$.
Note that 
$M(\Gamma, h_\epsilon) = M_0 \cup M_1$.
By the induction hypothesis there is a sequence of smooth functions 
$(f_k^0)_{k=1}^\infty$ supported in $\B(\Gamma, \tilde h_\epsilon; T))$ 
such that 
\begin{align*}
W f_k^0 \to  1_{M_0} \psi, \quad k \to \infty.
\end{align*}
Moreover, by Lemma \ref{lem_density} there is a sequence of smooth functions 
$(f_k^1)_{k=1}^\infty$ supported in $\B(\Gamma_{J}, T_J; T))$ such that
\begin{align*}
W f_k^1 \to  1_{M_1} (\psi - 1_{M_0} \psi), \quad k \to \infty.
\end{align*}
Thus $W (f_k^0+f_k^1) \to \psi$.
This proves that the density holds for $h_\epsilon$.

Suppose now that $\psi \in L^2(M(\Gamma, h);E)$.
We have shown that there is a smooth function 
$f$ supported in $\B(\Gamma, h_\epsilon; T)$ such that 
\begin{align*}
\norm{1_{M(\Gamma, h_\epsilon)} \psi - Wf}_{L^2(M;E)}^2 < \epsilon.
\end{align*}
Thus
\begin{align*}
\norm{ \psi - Wf}_{L^2(M;E)}^2
&< \epsilon + \left(\int_{M(\Gamma, h)} |\psi|_E^2 dx - \int_{M(\Gamma, h_\epsilon)} |\psi|_E^2 dx\right).
\end{align*}
The Riemannian volumes converge $|M(\Gamma, h_\epsilon)| \to |M(\Gamma, h)|$ as $\epsilon \to 0$, see \cite[Lemma 4.3]{Oksanen2011}.
Thus the claimed density holds.
\end{proof}

\section{Local reconstruction near the boundary}

In this section we show how to recover the coefficients of $P+V$, up to the gauge invariances,
near the accessible part of the boundary $\S$
given the map $\Lambda_\S^{2T}$.
The main novelty is the recovery of the connection and potential by using such sources $f$ that $Wf$ localizes near a point in $M$.
The basic idea of finding localized $Wf$ 
given $\Lambda_\S^{2T}$ is described in Lemma \ref{lem_orthogonal}, and the inner products appearing in this lemma are shown to be determined by $\Lambda_\S^{2T}$ in Corollary \ref{cor_test_convergence}.
The localization technique is refined in Lemma \ref{lem_conv_to_delta}, and the localized solutions are then used to probe the connection and potential in the proof of Theorem \ref{th_local_connection}.

\subsection{Inner products}

We begin by generalizing an integration by parts technique 
due to Blagovestchenskii in the $1+1$ dimensional scalar valued case \cite{Blagovescenskiui1971}.
For a multidimensional scalar valued case this was first used by Belishev  \cite{Belishev1987}.

\begin{Lemma}
\label{lem_Blago}
Let $T > 0$, let $\S \subset \p M$ be open, and let $f$ and $h$ be functions in $C_0^\infty((0,2T) \times \S;E)$.
Then
\begin{align*}
&\pair{W f, Wh}_{L^2(M; E)}
\\&\qquad= \pair{f, J \Lambda_\S^{2T} h}_{L^2((0,2T) \times \S;E)} - \pair{f, (\Lambda_\S^{2T})^* J h}_{L^2((0,2T) \times \S;E)},
\end{align*}
where $J$ is the integral operator in the time variable with the kernel $\mbox{\rm sgn}(t-s) 1_L(t,s)/4$.
Here 
$L=\{(s,t) \in \R^2:\;0\leq t+s\leq 2T,\;t,s>0\}$.
\end{Lemma}
\begin{proof} 
We write $u^f = u$ for the solution of (\ref{eq:wave})
and define the function $w(t,s) = \langle u^{f}(t),u^{h}(s)\rangle_{L^{2}(M;E)}$. We have
\begin{align*}
&(\partial_{t}^{2}-\partial_{s}^{2})w(t,s) =\langle \partial_{t}^{2}u^{f}(t),u^{h}(s)\rangle_{L^{2}(M;E)}-\langle u^{f}(t),\partial_{s}^{2}u^{h}(s)\rangle_{L^{2}(M;E)}
\\
&\quad=-\langle \nabla^*\nabla u^{f}(t),u^{h}(s)\rangle_{L^{2}(M;E)}+\langle u^{f}(t),\nabla^*\nabla u^{h}(s)\rangle_{L^{2}(M;E)}
\\
&\quad=-\int_{\partial M}\langle \nabla_{\nu} u^{f}(t),u^{h}(s)\rangle_{E}\,dS 
+ \int_{\partial M}\langle u^{f}(t), \nabla_{\nu} u^{h}(s)\rangle_{E}\,dS\\
&\quad=
\int_{\partial M}\langle f(t), \Lambda^{2T}_\S h(s)\rangle_{E}\,dS
-\int_{\partial M}\langle \Lambda^{2T}_\S f(t), h(s)\rangle_{E}\,dS.
\end{align*}
Since $w(0,s)=w(t,0)=\partial_{t}w(0,s)=\partial_{s}w(0,s)=0$ and $w$ solves the above $1+1$ dimensional wave equation, the result follows by considering $w(T,T)$.
\end{proof}

\begin{Corollary}
\label{cor_test_convergence}
Let $T > 0$, $\S \subset \p M$ be open.
Then $\Lambda_\S^{2T}$ determines the inner products
\begin{align}
\label{the_inner_products}
\pair{Wf, Wh}_{L^2(M;E)}, \quad f,h \in C_0^\infty((0,2T) \times \S;E).
\end{align}
Moreover, $\Lambda_\S^{2T}$ determines,
for all $(f_j)_{j=1}^\infty \subset C_0^\infty((0,2T) \times \S;E)$, if 
the sequence $(Wf_j)_{j=1}^\infty$
converges, in the strong or weak sense, in $L^2(M;E)$.
\end{Corollary}
\begin{proof}
We allow the metric tensor $g$ to be {\it a priori} unknown on $\S$. 
However,
 $\Lambda_\S^{2T}$ determines the distances $d_g(x,y)$, $x,y \in \S$,
see e.g. \cite[Section 2.2]{Dahl2009},
and these distances determine $g$ on $\S$.
Thus we can assume without loss of generality that the Riemannian volume measure $dS$ of $(\S, g)$ is known, and Lemma 
\ref{lem_Blago} implies that $\Lambda_\S^{2T}$ determines the inner products (\ref{the_inner_products}).

For the second claim, we observe that the inner products (\ref{the_inner_products}) can be used to determine if
$(Wf_j)_{j=1}^\infty$ is a Cauchy sequence in $L^2(M;E)$.
This allows us to determine if $(Wf_j)_{j=1}^\infty$
converges in the strong sense.
Moreover, using again (\ref{the_inner_products})
we can determine if $(Wf_j)_{j=1}^\infty$ is bounded in $L^2(M;E)$,
and we may test the weak convergence 
analogously to \cite[Lemma 3]{Lassas2012}.
\end{proof}


\subsection{Reconstruction of the metric tensor}

Our reconstruction of the metric tensor is based on the proof
in \cite{Lassas2012}. 
The following lemma is a variation of \cite[Lemma 6]{Lassas2012}.
We give a short proof for the convenience of the reader.

\begin{Lemma}
\label{lem_domi_test}
Let $T > 0$, $s \in (0,T]$, let $\Sigma, \Gamma \subset \p M$ be open and let $h : \Gamma \to [0,T]$.
Suppose that $h$ is of form (\ref{def_h}).
Then the following are equivalent:
\begin{itemize}
\item[(i)] $M(\Sigma, s) \subset M(\Gamma, h)$.
\item[(ii)] For all $f_0 \in C_0^\infty(\B(\Sigma, s; T); E)$
there is a sequence $(f_j)_{j=1}^\infty$ in $C_0^\infty(\B(\Gamma, h; T); E)$
such that $(W(f_0 - f_j))_{j=1}^\infty$ 
converges to zero in $L^2(M;E)$.
\end{itemize}
\end{Lemma}
\begin{proof}
The implication from (i) to (ii) follows from the density of (\ref{density_h_Wf}) in $L^2(M(\Gamma, h); E)$.
We will now show that (ii) implies (i).
We denote 
\begin{align*}
&M_0 = M(\Sigma, s),
\quad 
M_1 = M(\Gamma, h),
\\
&S_0 = \B(\Sigma, s; T), 
\quad 
S_1 = \B(\Gamma, h; T).
\end{align*}
Let us assume that (i) does not hold.
There is a nonempty open set $U \subset M_0$ such that $U \cap M_1 = \emptyset$,
see \cite[Lemma 6]{Lassas2012}.
By Lemma \ref{lem_density} there is a smooth function 
$f_0$ supported in $S_0$ such that $\int_U Wf_0 dx \ne 0$.
However, by finite speed of propagation $Wf|_U = 0$ for any 
$f$ supported in $S_1$. Thus 
\begin{equation*}
\pair{W (f_0 - f), 1_U}_{L^2(M;E)} = \pair{Wf_0, 1_U}_{L^2(M;E)} \ne 0,
\end{equation*}
for all $f$ supported in $S_1$ and (ii) does not hold.
\end{proof}

By Corollary \ref{cor_test_convergence}
we can determine, given the restricted Dirichlet-to-Neumann map  $\Lambda_\S^{2T}$, whether the condition (ii) in Lemma \ref{lem_domi_test} holds for a function $f_0$ and a sequence
$(f_j)_{j=1}^\infty$, assuming that $\Sigma, \Gamma \subset \S$.

\begin{Remark}
\label{rem_nonsymm}
Suppose for the moment that we weaken the symmetry assumptions 
by not requiring (\ref{symm_V}). Then the closest analogue of the identity in Lemma \ref{lem_Blago} allows us to compute the inner products 
$\pair{\tilde u(T), Wh}_{L^2(M;E)}$ 
where $\tilde u$ is the solution of \ref{eq:wave}
with $V$ replaced by its formal adjoint $V^*$.
It seems to be difficult to use such inner products to test for convergence as in the condition (ii). 
In the non-symmetric scalar valued case \cite{Kurylev2000a},
a global condition on the billiard flow of $(M,g)$
is assumed in order for the map 
$$
W : L^2((0,T) \times \S) \to L^2(M),
$$
not only to have a dense range, but to be surjective. 
In this case, it is easy to test for a variant of the condition (ii)
where convergence in the norm is replaced by weak convergence.
In \cite{Lassas2012} a similar difficulty is treated by imposing an asymptotic spectral condition of the type that was first studied in \cite{Hassell2002}.
\end{Remark}

Let $\Gamma \subset \p M$ be open and let $T > 0$.
We recall that the cut distance $\sigma_\Gamma$ is defined by (\ref{def_cut_distance}),
and define 
\begin{align}
\label{31.1}
&\sigma_\Gamma^T(y) = \min(\sigma_\Gamma(y), T), \quad y \in \Gamma,
\\ \nonumber
&M_\Gamma^T = \{\gamma(s;y,  \nu);\ y \in \Gamma,\ s \in [0, \sigma_\Gamma^T(y))\}.
\end{align}

\begin{Theorem}
\label{th_local_metric}
Let $T > 0$ and let $\Gamma \subset \p M$ be open. 
Then the Riemannian manifold $(\Gamma, g)$, the Hermitian vector bundle $E|_\Gamma$ and 
$\Lambda_\Gamma^{2T}$ determine 
$(M_\Gamma^T, g)$.
\end{Theorem}
\begin{proof}
By combining Corollary \ref{cor_test_convergence} and Lemmas \ref{lem_density_h} and \ref{lem_domi_test}
we can determine the relation 
\begin{align}
\label{domi_incldata}
\{ (\Sigma, s, h);\ M(\Sigma, s) \subset M(\Gamma, h) \}
\end{align}
for any open  $\Sigma \subset \Gamma$, $s \in(0,T]$ and 
a function $h$ of form (\ref{def_h}).
This relation determines $\sigma_\Gamma^T$ and the Riemannian manifold $(M_\Gamma^T, g)$
by using the purely geometric method described in Sections 4.2-4.4
of \cite{Lassas2012}.
Note that the relations with $M(\Gamma, h)$ replaced by the union of two domains of influence 
are obtained by using piecewise continuous functions $h$ as in  \cite[Lem. 6]{Lassas2012}, 
and that also the two limiting arguments in the proof of \cite[Prop. 2]{Lassas2012} are needed.
\end{proof}


\subsection{Reconstruction of the connection}
\def\B{\mathcal B}
\def\E{\mathcal E}

Our reconstruction method is based on a use 
of sequences of sources $(f_j)_{j=1}^\infty$ such that 
$\supp(W f_j)$ converges to a point. 

\begin{Lemma}
\label{lem_orthogonal}
Let $\Gamma_1, \Gamma_2 \subset \p M$ be open and $r_1, r_2 > 0$.
Suppose that for a sequence $(f_j)_{j=1}^\infty \subset C_0^\infty((T-r_1,T) \times \Gamma_1; E)$
the sequence $(W f_j)_{j=1}^\infty$
converges weakly to a function $\phi \in L^2(M;E)$, and that 
$$
\pair{W f_j, W h}_{L^2(M;E)} \to 0, \quad h \in C_0^\infty((T-r_2, T) \times \Gamma_2; E).
$$
Then $\supp(\phi) \subset M(\Gamma_1, r_1) \setminus M(\Gamma_2, r_2)^\inter$.
\end{Lemma}
\begin{proof}
The lemma follows immediately from the density of the set (\ref{set_uT}).
\end{proof}



\begin{Lemma}
\label{lem_basis}
Let $T > 0$, $\Gamma \subset \p M$ be open,
and let $x \in \Gamma \cup M^\inter$ satisfy $d_g(x,\Gamma) < T$. 
Then there are functions $h_\ell \in C_0^\infty((0, 2T) \times \Gamma; E)$
such that $W h_\ell(x)$, $\ell=1,\dots,n$, form an orthonormal basis of the fiber $E_x$ of $E$ at $x$. 
\end{Lemma}
\begin{proof}
If $x \in \Gamma$, then $Wh(x) = h(T,x)$ and the claim clearly holds in this case. Suppose now that $x \in M^\inter$. 
It is enough to show that the fiber $E_x$ is spanned by the vectors
$$
Wh(x), \quad h \in C_0^\infty((0, T) \times \Gamma; E).
$$
In order to show this it is enough to show that if $e \in E_x$ and
\begin{align}
\label{innerp0_lem_basis}
\pair{e, Wh(x)}_E = 0, \quad h \in C_0^\infty((0, T) \times \Gamma; E),
\end{align}
then $e = 0$.

We recall that the adjoint of $W$ is given by $W^* \phi =\nabla_\nu v|_{(0,T) \times \p M}$, where $v$ is the solution of (\ref{adjoint_of_W}),
and that the continuity (\ref{Wstar_continuity}) holds. 
We choose $\phi = e \delta_x$. 
The restriction 
$W^* \phi|_{(0,T) \times \Gamma} =\nabla_\nu v|_{(0,T) \times \Gamma}$ vanishes by (\ref{innerp0_lem_basis}),
and $v|_{(0,T) \times \Gamma}$ vanishes by the boundary condition in (\ref{adjoint_of_W}).
We extend $v$ on the time interval $(0,2T)$ by the odd reflection with respect to $t=T$, and denote the extension still by $v$.
The extension satisfies $(\p_t^2 + P+V)v = 0$ on $(0,2T) \times M$.
Theorem \ref{th_uniq_cont} implies that $e = 0$.
\end{proof}

\begin{Lemma}
\label{lem_smoothness_test}
Let $\Gamma \subset \p M$ be open, let $T > 0$,
and let $e : M \to E$ be a section of $E$.
Let $U \subset M^\inter \cup \Gamma$ be open in $M$
and suppose also that $U \subset M(\Gamma, T)$.
Suppose, furthermore, that 
$
x \mapsto \pair{e(x), Wh(x)}_{E}
$
is smooth on $U$ for all $h \in C_0^\infty((0, 2T) \times \Gamma; E)$.
Then $e$ is smooth on $U$.
\end{Lemma}
\begin{proof}
Let $x \in U$, and let us choose $h_\ell$, $\ell = 1, \dots, n$, as in Lemma \ref{lem_basis}.
Then the functions $W h_\ell$ form a smooth frame near $x$,
and the representation of $e$ in this frame is smooth. 
\end{proof}

We recall that $|X|$ denotes the Riemannian volume of a measurable set  $X \subset M$,
and that the set $M_\Gamma$ is defined by (\ref{semigeod_strip}).

\begin{Lemma}
\label{lem_conv_to_delta}
Let $\Gamma \subset \p M$ be open.
Let $x \in M_\Gamma$, and let $y \in \Gamma$ and $s \in [0,\sigma_\Gamma(y))$ satisfy $\gamma(s; y, \nu)$. Define $s_k = s + 1/k$,
$$
Y_k = \{\tilde y \in \Gamma;\ d_g(\tilde y, y) < 1/k \},
\quad X_k = M(Y_k, s_k) \setminus M(\Gamma, s).
$$
Suppose that a double sequence 
$\Phi = (f_{jk})_{j,k=1}^\infty$
of functions in the space
$C_0^\infty((T-s_k,T) \times Y_k; E)$ 
satisfies the following
\begin{itemize}
\item[(i)] For each $k=1,2,\dots$, the sequence $(W f_{jk})_{j=1}^\infty$ converges weakly in $L^2(M;E)$
to a function supported in $X_k$.
\item[(ii)] There is $C > 0$ such that 
$$
\norm{W f_{jk}}_{L^2(M;E)} \le C |X_k|^{-1/2} , \quad j,k=1,2,\dots.
$$
\item[(iii)] The limit 
$\lim_{k \to \infty} \lim_{j \to \infty} \pair{W f_{jk}, Wh}_{L^2(M;E)}$
exists for any function $h$ in the space $C_0^\infty((0,2T) \times \Gamma; E)$.
\end{itemize}
Then there is a vector $e(x;\Phi) \in E_x$ that depends on $x$ and $\Phi$
such that 
\begin{align}
\label{conv_to_delta}
\lim_{k \to \infty} \lim_{j \to \infty} \pair{W f_{jk}, \phi}_{L^2(M;E)}
= \pair{e(x;\Phi), \phi(x)}_{E}, \quad \phi \in C^\infty(M; E).
\end{align}
\end{Lemma}

Note that we allow here the case $x \in \Gamma$, i.e. $s=0$.

\begin{proof}
By Lemma \ref{lem_basis} there are $h_\ell$
such that $Wh_\ell(x)$, $\ell=1,\dots,n$, 
form an orthonormal basis of $E_x$.
Let us write  $b_\ell= W h_\ell$ and denote
the weak limit of $(W f_{jk})_{j=1}^\infty$
by $u_k$.
We choose local coordinates $\tilde x$ in a neighborhood $U \subset M$ 
of $x$,
and suppose that $k$ is large enough so
that $X_k \subset U$
and that the sections $b_\ell(\tilde x)$ form a basis in $E_{\tilde x}$ for all $\tilde x \in X_k$.
Let $\phi \in C^\infty(M; E)$ and 
write $\phi(\tilde x) = c^\ell b_\ell(\tilde x) + (x^p - \tilde x^p)\psi_p(\tilde x)$,
where $c^\ell \in \C$ and $\psi_p \in C^\infty(U; E)$, $p=1,\dots, m$.
Then 
\begin{align}
\label{R_lem_conv_to_delta}
\pair{u_k, \phi}_{L^2(M;E)} = \overline{c^\ell} \pair{u_k, b_\ell}_{L^2(M;E)} + R_k,
\end{align}
where the remainder term satisfies
\begin{align*}
|R_k| &\le m \max_{p=1,\dots,m}\norm{\psi_p}_{C(U)} \diam(X_k) \int_{X_k} |u_k(\tilde x)|_E d\tilde x
\\&\le m \max_{p=1,\dots,m} \norm{\psi_p}_{C(U)} \diam(X_k)\, \norm{u_k}_{L^2(M;E)} |X_k|^{1/2}.
\end{align*}
Note that $\diam(X_k) \to 0$ since $X_k \supset X_{k+1}$ and $X_k \to x$ as $k \to \infty$.
Thus (ii) implies that $R_k \to 0$.
By (iii) the limits 
$$
a^\ell = \lim_{k \to \infty} \pair{u_k, b_\ell}_{L^2(M;E)},
\quad \ell=1,\dots,n,
$$
exist. 
We set $e = a^\ell b_\ell(x)$. Then 
$$
\lim_{k \to \infty} \pair{u_k, \phi}_{L^2(M;E)}
= \overline{c^\ell} \lim_{k \to \infty} \pair{u_k, b_\ell}_{L^2(M;E)}
= \sum_{\ell = 1}^n  a^\ell \overline{c^\ell} = \pair{e,\phi(x)}_{E}.
$$
\end{proof}

\begin{Lemma}
\label{lem_existence_of_es}
Let $\Gamma \subset \p M$ be open,
let $x \in M_\Gamma$ and let $e \in E_x$.
Then there is a double sequence $\Phi=(f_{jk})_{j,k=1}^\infty$ 
that satisfies the conditions of  
Lemma \ref{lem_conv_to_delta}, and furthermore, $e(x; \Phi) = e$ where $e(x; \Phi)$ is as in (\ref{conv_to_delta}).
\end{Lemma}
\begin{proof}
Let $\tilde e \in C^\infty(M;E)$ satisfy $\tilde e(x) = e$.
By Lemma \ref{lem_density} there is a double sequence 
$\Phi=(f_{jk})_{j,k=1}^\infty$ 
of functions in
$C_0^\infty((T-s_k,T) \times Y_k; E)$ such that 
$(W f_{jk})_{j=1}^\infty$ converges to the function $u_k = |X_k|^{-1} 1_{X_k} \tilde e$.
We recall that $1_{X_k}$ 
is the indicator function of the set $X_k$
and $|X_k|$ is its volume. 
Moreover, $u_k$ satisfies $\norm{u_k}_{L^2(M;E)} \le |X_k|^{-1/2}\norm{\tilde e}_{L^\infty(M;E)}$ 
and, 
for a function $\phi \in C^\infty(M;E)$,
$$
\pair{u_k, \phi}_{L^2(M;E)} =
\frac{1}{|X_k|} \int_{X_k} \pair{\tilde e(\tilde x), \phi(\tilde x)}_E d\tilde x
\to \pair{e, \phi(x)}_{E}, 
$$
where $\tilde x$ are local coordinates on $X_k$.
\end{proof}

\begin{Theorem}
\label{th_local_connection}
Let $T > 0$, let $\Gamma \subset \p M$ be open and suppose that the vector bundle 
$E|_\Gamma$
is trivial. Then the Riemannian manifold $(M_\Gamma^T, g)$, where $M_\Gamma^T$ is defined in (\ref{31.1}),
the Hermitian vector bundle $E|_\Gamma$ and the restricted Dirichlet-to-Neumann map $\Lambda_\Gamma^{2T}$
determine the Hermitian vector bundle $E|_{M_\Gamma^T}$, the connection $\nabla$ and the potential $V$ on $E|_{M_\Gamma^T}$.
\end{Theorem}
\begin{proof}
We choose for each $x \in M_\Gamma^T$ a double sequence $\Phi^x=(f^x_{jk})_{j,k=1}^\infty$
satisfying conditions (i)--(iii) of Lemma \ref{lem_conv_to_delta}. Observe that,
by combining Corollary \ref{cor_test_convergence} and Lemma \ref{lem_orthogonal},
we can determine if condition (i) of Lemma \ref{lem_conv_to_delta} is valid, while
 conditions (ii) and (iii) can be verified by using 
Lemma \ref{lem_Blago} alone.
We use Lemma \ref{lem_Blago} once again to compute the 
inner products
$\pair{e(x; \Phi^x), Wh(x)}_E$ for $h \in C_0^\infty((0,2T) \times \Gamma; E)$.
Next we will impose some further conditions on the choice of the double sequences $\Phi^x$.

First, we choose the double sequences $\Phi^x, \, x \in M_\Gamma^T$
so that the functions 
\begin{align}
\label{def_fun_e}
x \mapsto \pair{e(x; \Phi^x), Wh(x)}_{E}, \quad h \in C_0^\infty((0,2T) \times \Gamma; E),
\end{align}
are smooth in $M_\Gamma^T$.
Then Lemma \ref{lem_smoothness_test}
implies that $e(x) = e(x; \Phi^x)$ is a smooth section of the vector bundle $E|_{M_\Gamma^T}$.

Second, we pick an orthonormal frame $\B = (b_\ell)_{\ell=1}^n$
of $E|_{\Gamma}$ and choose double sequences 
$\Phi_\ell^x =(f^x_{jk, \ell})_{j,k=1}^\infty, \, \ell=1, \dots,n,$ so that
the corresponding smooth sections $e_\ell(x) = e(x;\Phi_\ell^x)$ satisfy,
$$
\pair{e_\ell(x), Wh(x)}_{E} = \pair{b_\ell(x), h(T,x)}_{E}, \quad x \in \Gamma,\ 
h \in C_0^\infty((0,2T) \times \Gamma; E).
$$
This condition implies that $e_\ell = b_\ell$ on $\Gamma$.

Our next goal is to choose $\Phi^x_\ell$ so that the
corresponding sections $e_\ell$  form an orthonormal frame also on the set
$M_0 = M_\Gamma^T \cap M^\inter$.
To this end, we observe  that the  vector bundle $E|_{M_\Gamma^T}$ is trivial. This follows from
\cite[Th. 4.2.4]{Hirsch1994}, since the identity map on $M_\Gamma^T$ is smoothly homotopic with
the map $(s,y) \mapsto (0,y)$
in  coordinates (\ref{semigeodesic_coord}).

Let $x \in M_0$, and choose 
a cut off function $\chi \in C_0^\infty(M_0)$
such that $\chi(x) = 1$.
As the functions (\ref{def_fun_e}) and 
the geometry $(M_\Gamma^T, g)$ are known, we can compute
the limits
\begin{align}
\label{towards_ek_el}
\lim_{k \to \infty} \lim_{j \to \infty}\pair{\chi e_\kappa, W f_{jk, \ell}^x}_{L^2(M;E)}
= \pair{e_\kappa(x), e_\ell(x)}_E,
\quad \kappa, \ell = 1,\dots,n,
\end{align}
where the equality follows from Lemma \ref{lem_conv_to_delta}. Hence we can choose the double sequences $\Phi^x_\ell$ so that $\E = (e_\ell)_{\ell=1}^n$ forms an orthonormal frame on $M_0$.
Note that Lemma \ref{lem_existence_of_es} implies that 
for any frame on $M_0$ there are double sequences $\Phi^x_\ell$, $\ell = 1,\dots,n$, $x \in M_0$,
such that the corresponding functions $e_\ell$ coincide with the frame.

Now 
$(x,a) \mapsto a^\ell e_\ell(x)$,
where $a = (a^\ell)_{\ell=1}^n \in \C^n$ and $x \in M_\Gamma^T$,
is a trivialization of $E|_{M_\Gamma^T}$,
and the Hermitian inner product is given by 
$$
\pair{a^\ell e_\ell(x), c^\kappa e_\kappa(x)}_E = \sum_{\ell = 1}^n
a^\ell \overline{c^\ell}, \quad a,c \in \C^n,\ x \in M_\Gamma^T,
$$
on this trivialization.

Let us write $u^h = u$ for the solution of (\ref{eq:wave}) with $f=h$.
The functions (\ref{def_fun_e}) determine the representation of 
\begin{align}
\label{Wh_repr_M_Gamma}
Wh(x) = u^h(t,x), \quad t=T,\ x \in M_0,\ h \in C_0^\infty((0,2T) \times \Gamma; E),
\end{align}
in the frame $\E$. 
To avoid cumbersome notation, we will not make explicit distinction  
between the functions (\ref{Wh_repr_M_Gamma})
and their representation until Section \ref{sec_gluing}.

Observe that the wave equation (\ref{eq:wave}) is translation invariant in time
in the sense that $u^h(t-s,\cdot) = u^{\tilde h}(t,\cdot)$
where $\tilde h(t,\cdot) = h(t-s,\cdot)$ and $s \ge 0$.
Thus the functions (\ref{Wh_repr_M_Gamma}) are determined also for $t \in (0,T)$.
We differentiate twice in time and obtain the functions
\begin{align*}
(P+V) u^h(t,x), \quad t \in (0,T),\ x \in M_0,\ h \in C_0^\infty((0,2T) \times \Gamma; E).
\end{align*}
Let $\phi \in C_0^\infty(M_0; E)$. We can compute the inner products
$$
\pair{(P+V) u^h(T), \phi}_{L^2(M;E)}
= \pair{Wh, (P+V) \phi}_{L^2(M;E)},
$$
for $h \in C_0^\infty((0, 2T) \times \Gamma; E)$.
As the functions (\ref{Wh_repr_M_Gamma}) are known and dense in $L^2(M_0;E)$, we can determine $(P+V)\phi$
on $M_0$. 

Let $x \in M_0$, $\ell = 1, \dots, n$ and $k=1,\dots, m$.
We choose $\phi=\phi^k_\ell$ such that $\phi(x) = 0$
and $\p_j \phi(x) = \delta_{j}^k e_\ell$ for $j=1,\dots,m$. 
As the metric tensor is known near $x$, we can compute $d^* d \phi$ at $x$. 
Thus we can recover the first order term in $(P+V) \phi$ at $x$.
By (\ref{eq:P}), this is
$$
-2(A, d\phi)(x) = -2 g^{ik}(x) A_i e_\ell(x),
$$ 
and therefore $A$ can be determined. 
Finally, $A$ and $g$ determine $P$, and we can determine $V$ by 
$V = P + V - P$.
\end{proof}

\subsection{Reconstruction of $\nabla$ when $(M,g)$ is known and $E$ is trivial}

We will show next that Corollary \ref{cor:test}
follows from the above local reconstruction step, 
that is, from the proof of Theorem \ref{th_local_connection}.

\def\E{\mathcal E}
\begin{Corollary}
Suppose that $(M,g)$ is known, $E$ is the trivial bundle $M \times \C^n$,
and that $T > \max_{x \in M} d_g(x, \p M)$.
Let $d_A$ be a Hermitian connection on $E$.
Then the Dirichlet-to-Neumann map $\Lambda_{\p M; A}^{2T}$
determines the orbit
$$
\mathcal O(A) = \{U^{-1} A U + U^{-1} d U;\ U : M \to U(n),\ U|_{\p M} = Id \}.
$$
\end{Corollary}
\begin{proof}
Let $b_1, \dots, b_n$ be the standard basis of $\C^n$
and let $\B$ be the corresponding constant frame of $E$.
Let $\E$ be the orthonormal frame of 
$E|_{M_{\p M}}$ chosen in the proof of Theorem \ref{th_local_connection}. 
We recall that $\E$ can be enforced to satisfy $\E = \B$
on $\p M$.

We have $M_{\p M} = M \setminus N$ where the cut locus $N$ is of measure zero, see e.g. \cite{Chavel2006}.
In particular, $M_{\p M}$ is dense in $M$.
We know the representation 
of the functions $Wh$, $h \in C_0^\infty((0,2T) \times \p M; E)$,
in the frame $\E$, see (\ref{Wh_repr_M_Gamma}) above.
Let us impose the further condition on the 
choice of $\Phi^x_\ell$ in the proof of Theorem \ref{th_local_connection}
 that the representation of $Wh(x)$ in the frame $\E$ is smooth in $M = \overline{M_{\p M}}$
for all $h \in C_0^\infty((0,2T) \times \p M; E)$.
Then Lemma \ref{lem_smoothness_test} implies that $\E$
gives a smooth frame for the whole vector bundle $E$.

There is a smooth transition function $U : M \to U(n)$ 
between the two frames $\E$
and $\B$,
and $U = Id$ on $\p M$. 
Moreover, we can reconstruct the representation of $d_A$ in the frame $\E$.
Let us denote the representation by $d_{\widetilde A}$. Then 
$$
\widetilde A = U^{-1} A U + U^{-1} dU,
$$
and hence we can determine the orbit $\mathcal O(\widetilde A) = \mathcal O(A)$.
\end{proof}

Suppose now that $d_A$ and $d_B$ are two Hermitian connections on $E$,
and that the assumptions of Corollary \ref{cor:test} are satisfied. 
Then the above corollary implies that $\mathcal O(A) = \mathcal O(B)$,
and we have shown Corollary \ref{cor:test}.

\section{Global reconstruction}

In this section we show how to recover globally the coefficient of $P+V$, up to the gauge invariances,
by iterating the local reconstruction step 
and by continuing the data $\Lambda_\S^{2T}$
inside the region that we have already reconstructed.
We will begin by giving a brief outline of the iterative scheme.
The data $\Lambda_\S^{2T}$ can be viewed as a model of measurements with sources and receivers on $\Gamma$.
To initialize the iteration, we choose a small ball $B_0$ in the region
where the coefficients of $P+V$ are already known from the local reconstruction step in the previous section.
Then we use unique continuation to recover data modelling 
measurements with sources on $\Gamma$ and receivers on $B_0$,
and also with both sources and receivers on $B_0$.
Then we repeat the local reconstruction step for the data with sources and receivers on $B_0$,
and recover the coefficients of $P+V$ on a larger ball $B$
containing $B_0$.
Using unique continuation again, we recover the data with sources and receivers on a small ball $B_1$ in $B$, and also the data with sources on $\Gamma$ and the receivers on $B_1$.
Iterating this alternating procedure, we can cover $M$ with small patches where the 
coefficients of $P+V$ are known. 
The data with sources on $\Gamma$ and the receivers on $B_0$, $B_1$, \dots, is then used to glue the patches together. 

\subsection{Continuation of the data}
\label{sec_redatuming}
\def\BB{\mathbb B}

For $T > 0$ and open sets $B \subset M$ and $\Gamma \subset \p M$, we define the map
$$
L_{\Gamma, B}^{T} f = u|_{(0, T) \times B}, \quad f \in C_0^\infty((0,T) \times \Gamma; E),
$$ 
where $u$ is the solution of (\ref{eq:wave}).
Moreover, for open $B \subset M^\inter$, we define the map
$$
L_B^{T} F = u|_{(0,T) \times B}, \quad F \in C_0^\infty((0,T) \times B; E),
$$ 
where $u$ is the solution of 
\begin{align}
\label{eq_wave_inter}
&(\partial_{t}^{2}+P+V)u(t,x) = F, \quad (0,\infty)\times M,\\
&u|_{(0,\infty)\times\partial M}=0,
\notag\\
&u|_{t=0}=\partial_{t}u|_{t=0}=0.
\notag
\end{align}
We write $B(x,\e) = \{y \in M;\ d_g(y,x) < \epsilon\}$
for $x \in M$ and $\e > 0$.

\begin{figure}[t]
\def\svgwidth{6cm}
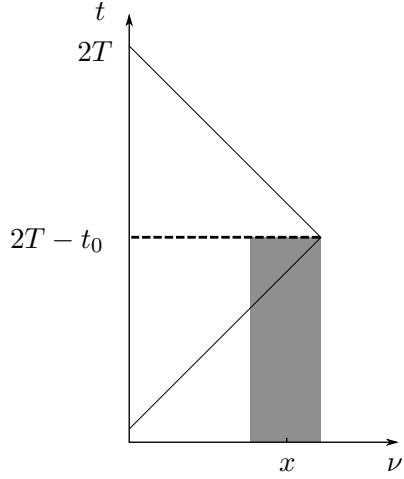
\caption{
A schematic of the unique continuation argument
in the proof of Lemma \ref{lem_local_move_boundary}.
The origin represents the set $\Gamma$,
and the gray area is the cylinder $(0,2T-t_0) \times B$.
In order to recover $u$ on $\{2T-t_0\} \times B$, data
$(u, \nabla_\nu u)$ is needed on the  
cylinder $I \times \Gamma$ where $I = (2T-2t_0,2T)$.
We may translate the interval $I$ to cover the whole gray cylinder. 
}
\label{fig_uniq_cont}
\end{figure}

\begin{Lemma}
\label{lem_local_move_boundary}
Let $T > 0$,  $\Gamma \subset \p M$ be open
and let 
$x \in M_\Gamma^{T}$.
Define $s = d_g(x,\Gamma)$, let $\e \in (0,T-s)$
and define
$$
B = B(x,\epsilon), \quad t_0 = s + \e.
$$
Then $\Lambda_\Gamma^{2T}$ and the structure $(g,E,\nabla,V)$ on $M_\Gamma^{T}$
determine the map $L_{\Gamma, B \cap M_\Gamma^{T}}^{2T-t_0}$.
Furthermore, if $B \subset M_\Gamma^{T} \cap M^\inter$ then they
determine also the map $L_{B}^{2(T-t_0)}$.
\end{Lemma}
\begin{proof}
Let $f \in C_0^\infty((0,2T) \times \Gamma; E)$.
We will next use unique continuation 
to determine
$L_{\Gamma, B\cap M_\Gamma^{T}}^{2T-t_0} f$
given $P+V$ on $M_\Gamma^{T}$ and $\Lambda_{\Gamma}^{2T} f$.
Let us first extend the solution $u$ of (\ref{eq:wave}) by $0$ to $(-\infty, 0) \times M$.
We denote the distance function of 
$(M_\Gamma^{T}, g)$ by $\tilde d_g$ 
and observe that
$$
\tilde d_g(x, \Gamma) = d_g(x,\Gamma),
\quad x \in M_\Gamma^{T},
$$
by the definition of $M_\Gamma$, see (\ref{semigeod_strip}).
Let $\tilde u$ be a solution 
of 
\begin{align}
\label{tildeu_uniqcont_eq}
(\partial_{t}^{2}+P+V) \tilde u = 0, \quad (-\infty,2T)\times M_\Gamma^T,
\end{align}
satisfying 
the boundary conditions
\begin{align}
\label{tildeu_uniqcont_bc}
\tilde u = f \quad \text{and} \quad \nabla_\nu \tilde u = \Lambda_{\Gamma}^{2T} f
\quad
\text{on}\ (-\infty,2T) \times \Gamma.
\end{align}
Given $P+V$ on $M_\Gamma^{T}$ and $\Lambda_{\Gamma}^{2T} f$,
we can determine the set of functions 
$$
\mathbb U_f = \{\tilde u \in C^\infty((-\infty,2T)\times M_\Gamma^T);\ \text{(\ref{tildeu_uniqcont_eq}) and (\ref{tildeu_uniqcont_bc}) hold}\}.
$$
Let $\tilde u \in \mathbb U_f$, and apply Theorem \ref{th_uniq_cont}
on the function $w = \tilde u - u$
with $M$ replaced by $M_\Gamma^{T}$ and with suitable translations in the time variable, see Figure \ref{fig_uniq_cont}.
This implies that $\tilde u = u$ on
$(0,2T-t_0) \times (B\cap M_\Gamma^{T})$,
and we have shown the first claim. 

Let us now assume that $B \subset  M_\Gamma^{T} \cap M^\inter$.
We will reconstruct the map $L_B^{2(T-t_0)}$ in two steps that we outline before giving a detailed proof.
Note that $L_{\Gamma, B}^{2T-t_0}$ can be interpreted as data
with sources on $\Gamma$ and receivers on $B$. 
We will first transpose $L_{\Gamma, B}^{2T-t_0}$
and obtain data with sources on $B$ and receivers on $\Gamma$.
Then we will use unique continuation to obtain data with both sources and receivers on $B$, that is, the map $L_B^{2(T - t_0)}$.

By taking the adjoint of $L_{\Gamma, B}^{2T-t_0}$ and conjugating it with the operator reversing the time on the interval 
$(0,2T-t_0)$, we get the map
\begin{align}
\label{data_B_to_S}
F \mapsto \nabla_\nu u : C_0^\infty((0, 2T -t_0) \times B;E)
\to C^\infty((0,2T -t_0) \times \Gamma;E),
\end{align}
where $u$ is the solution of (\ref{eq_wave_inter}). We extend $u$ by $0$ to $(-\infty, 0) \times M$, and let $\tilde u$ be a solution 
of 
$$
(\partial_{t}^{2}+P+V) \tilde u = F, \quad (0,\infty)\times M_\Gamma^T,
$$
satisfying 
$\tilde u = 0$
and $\nabla_\nu \tilde u =\nabla_\nu u$
on $(-\infty, 2T-t_0) \times \Gamma$.
Then $w=\tilde u-u$ satisfies conditions of Theorem \ref{th_uniq_cont} with $M$ again replaced by $M_\Gamma^T$,
and therefore $\tilde u=u$ 
on $(0,2T-t_0-t_0) \times B$.
This implies the second claim. 
\end{proof}

We denote by $SM$ the unit sphere bundle of $M$.
Similarly to $\sigma_\Gamma$ and $\sigma_\Gamma^T$, see (\ref{def_cut_distance}) and (\ref{31.1}),  
we define for $x \in M^\inter,\, \xi \in S_x M$ and $T > 0$,

\begin{align*} 
\sigma_x(\xi) &= \sup \{t \in (0,\tau_x(\xi)];\ d_g(\gamma(t;x,\xi), x) = t\}, 
\\ \nonumber
\tau_x(\xi) &= \sup \{t \in (0,\infty);\ \gamma(t;x,\xi) \in M^\inter\},
\end{align*}
and $\sigma_x^T(\xi) = \min(\sigma_x(\xi), T)$.
Moreover, we define 
$$
M_x^T=\{\gamma(t;x,\xi);\, \xi \in S_x M, t \in [0, \sigma_x^T(\xi)) \}.
$$
Note that the injectivity radius $\hbox{inj}_x$ at a point $x \in M^\inter$
satisfies 
$$
\hbox{inj}_x = \min_{\xi \in S_x M} \sigma_x(\xi).
$$

\begin{Lemma}
\label{lem_local_reconstruction}
Let $T > 0$, $x \in M^\inter$, 
$\epsilon \in (0, \hbox{inj}_x)$, and set
$
B = B(x, \epsilon)$.
Then $L_B^{2T}$ and the structure $(g,E,\nabla,V)$ on $B$
determine the structure $(g,E,\nabla,V)$ on $M_x^{T+\e}$.
\end{Lemma}
\begin{proof}
We define $\tilde M = M \setminus B$
and consider the wave equation
\begin{align}
\label{weq_Mtilde}
&(\partial_{t}^{2}+P+V)\tilde u=0, \qquad\qquad\qquad\quad (0,\infty)\times \tilde M,
\\\nonumber
&\tilde u|_{(0,\infty)\times\partial B}= f,
\quad
\tilde u|_{(0,\infty)\times\partial M}=0,
\\\nonumber
&\tilde u|_{t=0}=\partial_{t}\tilde u|_{t=0}=0.\notag
\end{align}
We will show that $L_B^{2T}$ determines the
restricted Dirichlet-to-Neumann map $\Lambda_{\p B}^{2T}$
of $\tilde M$, that is, the map 
\begin{align*}
\Lambda_{\p B}^{2T} f = \nabla_\nu \tilde u|_{(0,2T) \times \p B}, \quad f \in C^\infty_0((0,2T) \times \p B;E),
\end{align*}
where $\tilde u$ is the solution of (\ref{weq_Mtilde}).
Let $f \in C^\infty_0((0,2T) \times \p B;E)$ and extend
the solution of (\ref{weq_Mtilde}) smoothly into $(0, \infty) \times B$ keeping the notation $\tilde u$ for the extension.
Then $\tilde u$ satisfies (\ref{eq_wave_inter}) with
$\tilde F=(\p_t^2+P+V)\tilde u$,
and $\tilde F$ belongs to
\begin{align*}
\mathcal C
= \{F \in C^\infty((0,\infty) \times M;E);\ \supp(F) \subset (0, \infty] \times \overline B\}.
\end{align*}

Observe that that $L_B^{2T}$ has a unique extension as an operator on $L^2((0,2T) \times B;E)$.
By using this extension, we can determine the set 
$$
\mathbb F_f = \{F \in \mathcal C;\ 
L_B^{2T}F|_{(0,2T) \times \p B} = f
\}.
$$
Since the solution of (\ref{weq_Mtilde})
is unique, it hods for $F \in \mathbb F_f$ that 
$$
\nabla_\nu L_B^{2T}F|_{(0,2T) \times \p B} = \nabla_\nu \tilde u|_{(0,2T) \times \p B}.
$$
We have shown that the map $L_B^{2T}$ determines the map
$\Lambda_{\p B}^{2T}$.

We denote by $\sigma_{\p B}$
the cut distance on the manifold $\tilde M$
defined analogously to (\ref{def_cut_distance})
and define $\sigma_{\p B}^T(y) = \max(\sigma_{\p B}(y), T)$, $y \in \p B$.
Note that the vector bundle $E|_{\p B}$ is trivial,
in fact, $E$ is trivial over $M_x^T$ due to its contractibility via the radial geodesics emanating from $x$.
We apply Theorems \ref{th_local_metric} and \ref{th_local_connection}
with $M = \tilde M$ and $\Gamma = \p B$.
This gives us the structure $(g, E, \nabla,V)$ on 
$$\tilde M_{\p B}^T = \{\gamma(s;y, \nu);\ y \in \p B,\ s \in [0, \sigma_{\p B}^T(y))\}.$$
Note that $\sigma_x(\xi)= \sigma_{\p B}(y)+\e$,
where $y=\gamma(\e; x, \xi)$, 
and therefore
 $M_x^{T+\e}=B \cup\tilde M_{\p B}^{T}$. 
\end{proof}

\begin{Lemma}
\label{lem_local_move}
Let $T_0, \epsilon_0 > 0$, $x_0 \in M^\inter$, 
and define $B_0 =B(x_0, \e_0)$ and $M_0 = M_{x_0}^{T_0}$.
Let $x \in M_0 \setminus B_0$ and define $s = d_g(x, x_0)$.
Let $T > 0$ and let $\e \in (0, \hbox{inj}_{x})$
satisfy 
$$
\e < d_g(x, \p M_0), \quad \e < T - s + \e_0.
$$
Define $B_1 = B(x,\epsilon)$ and $t_1 = s + \e - \e_0$.
Then $L_{B_0}^{2T}$ and the structure $(g,E,\nabla,V)$ on $M_0$ determine the map $L_{B_1}^{2(T-t_1)}$.
Furthermore, for open $\Gamma \subset \p M$,
$L_{\Gamma, B_0}^{2T}$ and the structure $(g,E,\nabla,V)$ on $M_0$
determine the map $L_{\Gamma, B_1}^{2T- t_1}$.
\end{Lemma}
\begin{proof}
By the proof of Lemma \ref{lem_local_reconstruction}, $L_{B_0}^{2T}$
determines $\Lambda^{2T}_{\p B_0}$. 
It holds that $d_g(x, \p B_0)=s-\e_0$, and Lemma \ref{lem_local_move_boundary}
shows that $\Lambda^{2T}_{\p B}$ and the structure $(g,E,\nabla,V)$ on $M_0$ determine $L_{B_1}^{2(T-t_1)}$.
Finally, $L_{\Gamma, B_0}^{2T}$ determines $L_{\Gamma, B_1}^{2T- t_1}$
by a unique continuation argument similar to that in the proof of Lemma \ref{lem_local_move_boundary}.
\end{proof}

\subsection{Gluing local reconstructions in the interior}
\label{sec_gluing}

In this section we show the following theorem:
\begin{Theorem}
Let $\S \subset \p M$ be open and suppose that
\begin{align}
\label{assumption_T}
T > \max_{x \in M} d_g(x, \S).
\end{align}
Then the Hermitian vector bundle $E|_\S$ 
and the restricted Dirichlet-to-Neumann operator $\Lambda_\S^{2T}$
determine the smooth manifold $M^\inter$ and the structure $(g,E,\nabla,V)$ on $M^\inter$.
\label{thm:maint_interior}
\end{Theorem}

Up to this point we have avoided writing all the isomorphisms explicitly, but in this section the distinction between different representations is crucial.
Let us choose an open cover $\mathcal G_\S$ of $\S$
consisting of small enough sets $\Gamma \subset \S$  so that 
each $\Gamma$ is a coordinate neighborhood in $\p M$ and that
the vector bundle $E|_\Gamma$ is trivial.
Then we may choose an open set $Y_\Gamma \subset \R^{m-1}$ and a unitary trivialization 
\begin{align}
\label{triv_on_Gamma}
\xymatrix{
        E \ar[r]^{\phi_\Gamma} \ar[d] & Y_\Gamma \times \C^n \ar[d] \\
        \Gamma \ar[r]_{\psi_\Gamma}       & Y_\Gamma }
\end{align}  
By a unitary trivialization we mean that the diagram (\ref{triv_on_Gamma}) commutes, $\phi_\Gamma$ is a smooth bijection that is linear in fibers, and that the Hermitian structure is preserved, that is, $\phi_\Gamma^* \pair{\cdot, \cdot}_{\C^n} = \pair{\cdot, \cdot}_{E}$.

Starting from the representation of $\Lambda_\Gamma^{2T}$
on the trivialization (\ref{triv_on_Gamma}), the local reconstruction method in Section 3 
determines the cut distance $\sigma_\Gamma^T : \Gamma \to (0,T)$,
a metric tensor $g_\Gamma$ on $X_\Gamma^T$, and a connection $\nabla_\Gamma$ and potential $V_\Gamma$ on $X_\Gamma^T \times \C^n$, 
such that there is a unitary trivialization
\begin{align}
\label{triv_on_M_Gamma}
\xymatrix{
        E \ar[r]^{\tilde \Phi_\Gamma} \ar[d] & X_\Gamma^T \times \C^n \ar[d] \\
        M_{\Gamma}^T \ar[r]_{\tilde \Psi_\Gamma}       & X_\Gamma^T }
\end{align}  
satisfying $g = \tilde \Psi_\Gamma^* g_\Gamma$, $\nabla = \tilde \Phi_\Gamma^* \nabla_\Gamma$ and $V = \tilde \Phi_\Gamma^* V_\Gamma$.
Here 
$$
X_\Gamma^T = \{(s,y) \in \R^m;\ s \in [0, \sigma_\Gamma^T \circ \psi_\Gamma^{-1}(y)),\ y \in Y_\Gamma\}
$$
is the representation of $M_\Gamma^T$ in boundary normal coordinates,
and the restriction of $\tilde \Phi_\Gamma$ on the vector bundle $E|_\Gamma$ 
coincides with $\phi_\Gamma$.
We recall that $\sigma_\Gamma^T$ is defined by 
(\ref{31.1}).

We will next iterate the procedure in Section \ref{sec_redatuming}.
The initial step is the following:
\begin{itemize}
\item[1.] Given $\Lambda_\Gamma^{2T}$ and a representation of the structure $(g, E, \nabla, V)$ on $M_{\Gamma}^T$,
that is, $g_\Gamma$, $X_\Gamma^T \times \C^n$, $\nabla_\Gamma$ and $V_\Gamma$, 
we choose $(s_0,y_0) \in X_\Gamma^T$ 
and $\epsilon_0 > 0$
such that 
\begin{align}
\label{B_first}
B_0 = B(z_0, \epsilon_0) \subset M_\Gamma^{T} \cap M^\inter,
\end{align}
where  $z_0 = \tilde \Psi_\Gamma^{-1}(s_0,y_0) \in M_\Gamma^T$.
\end{itemize}
We invoke Lemma \ref{lem_local_move_boundary} to reconstruct the representations of 
$L^{2T-t_0}_{\Gamma, B_0}$ and $L_{B_0}^{2(T-t_0)}$ on the trivialization (\ref{triv_on_M_Gamma}). 
Here 
\begin{align}
\label{iteration_t0}
t_0=s_0+\e_0,
\end{align}
and we emphasize that we do not know the point $z_0 \in M$, only its representation $(s_0,y_0)$ in the boundary normal coordinates.

We iterate Lemmas \ref{lem_local_reconstruction} and \ref{lem_local_move} as follows:
\begin{itemize}
\item[2.] Given a representation of $L_{B_j}^{2(T-t_j)}$, where $B_{j}=B(z_{j}, \epsilon_{j})$,
we determine a representation of the structure 
$(g,E,\nabla,V)$
on the set $M_{j}=M_{z_{j}}^{T - t_{j}+\e_j}$.
\item[3.] We choose $s_{j+1} > 0$, $\xi_{j+1} \in S_{z_j} M$
and $\e_{j+1} > 0$ such that 
$$
B_{j+1} = B(z_{j+1}, \e_{j+1}) \subset M_j,
$$
where $z_{j+1} = \gamma(s_{j+1}; z_j, \xi_{j+1})$.
Again, we do not know $z_{j+1}$,
only its representation $(s_{j+1}, \xi_{j+1})$ in normal coordinates at $z_j$.
Given representations of $L_{B_{j}}^{2(T - t_{j})}$
and $L_{\Gamma,B_{j}}^{2T-t_{j}}$,
we determine representations of $L_{B_{j+1}}^{2(T - t_{j+1})}$
and $L_{\Gamma,B_{j}}^{2T-t_{j+1}}$,
where
\begin{align}
\label{iteration_tj}
t_{j+1} = t_{j} + s_{j+1} + \epsilon_{j+1} -\e_{j}.
\end{align}
\end{itemize}
We terminate the iteration after repeating the steps 2 and 3 a finite number of times denoted by $N = 0, 1, 2,\dots$. 
Note that we must satisfy the condition $t_j < T$ in each step of the iteration.

If $N=0$ then we do not need to satisfy the constraint (\ref{B_first}).
That is, we can use Lemma \ref{lem_local_move_boundary}
to reconstruct a representation of 
$L^{2T-t_0}_{\Gamma, B_0 \cap M_\Gamma^T}$
where $B_0 = B(z_0, \e_0)$, $z_0 \in M_\Gamma^T$
and $\e_0 \in (0,T-s_0)$.
In particular, for $y_0 \in \Gamma$ and for small enough $\e_0 > 0$ we can reconstruct a representation of 
$L^{2T-\e_0}_{\Gamma, C_0}$
where 
\begin{align} \label{11.5}
&C_0=\{\gamma(s;y,\nu);\ s \in (0, \e_0),\, 
y \in B_{\p}(y_0, \e_0) \},
\end{align}
and $B_{\p}(y_0, \e_0) = 
\{y \in \p M;\ d_g(y, y_0) < \e_0 \}$.

There are is a lot freedom in our iteration process. Namely, we can choose $N$, the points $z_j$ and the radii 
$\epsilon_j$ freely within the constraints of the iteration. 
Let $A_\Gamma$ denote the set of all choices that are allowed within the constraints of iteration
 when starting from $\Gamma \in \mathcal G_\S$. 
We define also the disjoint union 
$
A = \bigsqcup_{\Gamma \in \mathcal G_\S} A_\Gamma$.

We denote by $B_\alpha = B_{N(\alpha)}$
the set chosen in the last invocation of step 3 
in the iteration process $\a\in A_\Gamma$,
and use analogous notation for other chosen quantities. 
The iteration gives us a metric tensor $g_\alpha$, a connection $\nabla_\alpha$ and a potential $V_\alpha$ such that there is 
a unitary trivialization
\begin{align}
\label{triv_on_B_alpha}
\xymatrix{
        E \ar[r]^{\tilde \Phi_\alpha} \ar[d] & X_\alpha \times \C^n \ar[d] \\
        B_\alpha \ar[r]_{\tilde \Psi_\alpha}       & X_\alpha }
\end{align}  
satisfying $g = \tilde \Psi_\alpha^* g_\alpha$, 
$\nabla = \tilde \Phi_\alpha^* \nabla_\alpha$ 
and $V = \tilde \Phi_\alpha^* V_\alpha$.
Here $X_\alpha$ is the open ball 
of radius $\e_{N(\a)}$
in $\R^m$ with center at the origin,
and $\tilde \Psi_\alpha$ gives normal coordinates at $z_{N(\a)}$.
The iteration gives also the representation $L_\alpha$ of $L_{\Gamma, B_\alpha}^{2T-t_{N(\a)}}$
on the trivialization (\ref{triv_on_B_alpha}).

If the iteration is terminated immediately after the initial step (that is, $N(\alpha)=0$) we allow $B_\alpha$
to be also of the form (\ref{11.5}).

Let us show that the balls $B_\alpha$, $\alpha \in A_\Gamma$, cover $M(\Gamma, T)^\inter$
and that they separate points:
\begin{itemize}
\item[(G1)] For all distinct $z, z' \in M(\Gamma, T)^\inter$ there are $\alpha, \beta \in A_\Gamma$ such that $z \in B_\alpha$, $z' \in B_\beta$ and $B_\alpha \cap B_\beta = \emptyset$. 
\end{itemize}
\begin{proof}
Let $z \in M(\Gamma, T)^\inter$.
Then there is a shortest path $\gamma$ from $\bar \Gamma$ to $z$
having length strictly less than $T$.
The path $\gamma$ can be perturbed to get a broken geodesic $\tilde \gamma$ from $y \in \Gamma$ to $z$ having 
length strictly less than $T$.
Moreover, $\tilde \gamma$ can be chosen so that it intersects $\p M$
only at its starting point $y$.
Then the points $z_j$, $j=1,\dots,N$, can be chosen along $\tilde \gamma$.
Moreover, when $z_0$ is close to $\Gamma$ and the radius $\epsilon_N$ is chosen small enough, we have $t_N < T$. 
Indeed, by (\ref{iteration_t0}) and (\ref{iteration_tj}),
$$
t_N = \e_N + s_0 + \sum_{j=1}^{N} s_j,
$$
where $s_0=d_g(z_0, \Gamma)$ and $s_j=d_g(z_j, z_{j-1})$.
In particular, the balls $B_\alpha$, $\alpha \in A$,
form an open cover of $M(\Gamma, T)^\inter$.

Let $z' \in M(\Gamma, T)^\inter$ and suppose that $z' \ne z$.
We may choose the radius $\epsilon_N$ small enough so that $\epsilon_N < d_g(z,z')/2$,
and perform an analogous construction for $z'$.
This gives us disjoint balls as claimed.
\end{proof}

Note that the assumption (\ref{assumption_T})
does not imply that $M(\Gamma, T) = M$
since $\Gamma$ might be smaller than $\S$.
However, it implies that the sets $M(\Gamma, T)^\inter$, $\Gamma \in \mathcal G_\S$,
form an open cover of $M^\inter$, and therefore
the sets $B_\alpha$, $\alpha \in A$, form an open cover of $M^\inter$ by (G1).
We will show next how to glue together the local representations  
of $(g,E,\nabla,V)$ on the sets $B_\alpha$, $\alpha \in A$.

\begin{Lemma}
\label{lem_gluing_seq}
Let $T > 0$, 
$\Gamma \subset \p M$ be open, and suppose that $B \subset M^\inter$ is open and satisfies
$B \subset M(\Gamma, T)$.
Let $h \in C_0^\infty(B;E)$ and $s \in (0,T)$.
Then the maps $\Lambda_\Gamma^{2T}$ and $L_{\Gamma,B}^{2T-s}$ together with the structure $(g,E)$ on $B$  
determine the non-empty set 
\begin{align}
\label{set_conv_w}
\{
(f_j)_{j=1}^\infty \subset C_0^\infty((0,2T) \times \Gamma; E);\ \text{$\lim_{j \to \infty} W f_j = h$
in $L^2(M;E)$}\}.
\end{align}
\end{Lemma}
\begin{proof}
We expand the squared norm
\begin{align*}
&\norm{W f_j - h}_{L^2(M;E)}^2
\\&\qquad= \pair{W f_j, W f_j}_{L^2(M;E)} 
- 2 \mathop{Re} \pair{W f_j, h}_{L^2(M;E)} 
+ \pair{h, h}_{L^2(M;E)}, 
\end{align*}
and observe that $\Lambda_\Gamma^{2T}$
determines the first term on the right-hand side by Corollary \ref{cor_test_convergence},
$L_{\Gamma,B}^{2T-s}$ and $(g,E)$ on $B$ determine the second term,
and $(g,E)$ on $B$ determines the third term. 
To conclude we observe that Lemma \ref{lem_density} implies 
that the set (\ref{set_conv_w}) is non-empty.
\end{proof}

\begin{Lemma}
\label{lem_gluing}
Suppose that open $\S \subset \p M$ and $T > 0$
satisfy (\ref{assumption_T}).
Let $x_1, x_2 \in M^\inter$.
We have $x_1 = x_2$ if and only if for all sufficiently small $\epsilon > 0$ and any 
$h_1 \in C_0^\infty(B(x_1, \epsilon);E)$ 
there is $h_2 \in C_0^\infty(B(x_2, \epsilon);E)$
such that 
\begin{align}
\label{eq_gluing}
\langle h_1 - h_2, Wf \rangle_{L^2(M;E)} = 0, \quad f \in C_0^\infty((0,2T) \times \S;E).	
\end{align}
\end{Lemma}
\begin{proof}
Let us suppose that $x_1 \ne x_2$.
We choose small enough $\epsilon > 0$ so that the balls 
$B(x_j, \epsilon)$, $j=1,2$, are disjoint.
We choose non-zero $h_1\in C_0^\infty(B(x_1, \epsilon);E)$ and let 
$h_2 \in C_0^\infty(B(x_2, \epsilon);E)$ be arbitrary. 
Then $h_1 \ne h_2$ and 
Lemma \ref{lem_density} implies that there is $f \in C_0^\infty((0,2T) \times \S;E)$
satisfying 
$$\langle h_1 - h_2, Wf \rangle_{L^2(M;E)} \ne 0.$$
The other implication is trivial.
\end{proof}

\def\X{\mathcal X}
\def\U{\mathcal U}
Lemmas \ref{lem_gluing_seq} and \ref{lem_gluing} allow us to determine if
two points $x_i \in X_{\alpha_i}$, 
$\alpha_i \in A_{\Gamma_i}$, $\Gamma_i \in \mathcal G_\S$,
$i=1,2$, satisfy 
\begin{align}
\label{base_point_equivalence}
\tilde \Psi_{\alpha_1}^{-1}(x_1) = \tilde \Psi_{\alpha_2}^{-1}(x_2).
\end{align}
Indeed, let $\epsilon > 0$ be small, let 
$\tilde B_i$, $i=1,2$, be
the geodesic ball in $(X_{\alpha_i}, g_{\alpha_i})$ 
with center $x_i$ and radius $\epsilon$,
and let $\tilde h_i \in C_0^\infty(\tilde B_i;E)$.
Then using Lemma \ref{lem_gluing_seq}, 
we can find sequences 
$(f_j^i)_{j=1}^\infty \subset C_0^\infty((0,2T) \times \Gamma_i; E)$
such that $\lim_{j \to \infty} W f_j^i = h_i$ where $h_i = \Phi_{\alpha_i}^* \tilde h_i$.
Note that in order to apply Lemma \ref{lem_gluing_seq}
it is enough to know  $\Lambda_\Gamma^{2T}$ and 
the representations $L_{\alpha_i}$ and $g_{\alpha_i}$, $i=1,2$.
By Corollary \ref{cor_test_convergence}, we can compute 
\begin{align}
\label{eq_gluing_pullback}
\lim_{j \to \infty} \langle W f_j^1 - W f_j^2, Wf \rangle_{L^2(M;E)}
= \langle  h_1 -  h_2, Wf \rangle_{L^2(M;E)},
\end{align}
for all $f \in C_0^\infty((0,2T) \times \S;E)$.
Hence we can use (\ref{eq_gluing}) to determine if (\ref{base_point_equivalence}) holds. 

The equation (\ref{base_point_equivalence}) gives
an equivalence relation on the disjoint union $\tilde{\X} = \bigsqcup_{\alpha \in A} X_\alpha$ and we denote by $\mathcal X$ and $q : \tilde{\X} \to \X$
the corresponding quotient space and 
the canonical map. 
Moreover, we define the set $\U_\alpha = q(X_\alpha) \subset \X$ and the restriction $q_\alpha = q|_{X_\alpha}$, $\alpha \in A$.
We will show that $\X$ is a smooth manifold:
\begin{itemize}
\item[(G2)] The maps $q_\alpha : X_\alpha \to \U_\alpha$ are bijective,
and there is a unique Hausdorff topology and a complete atlas on $\X$
such that each $q_\alpha^{-1}$ is a coordinate system.
\end{itemize}
As we can determine if $x$ and $x'$ are equivalent given the data $\Lambda_\Gamma^{2T}$, we see that the smooth structure of $\X$ is determined. 
Let us show (G2) simultaneously with the following:
\begin{itemize}
\item[(G3)] Let us define a map $\Psi : M^\inter \to \X$
by 
$\Psi(z) = q \circ \tilde \Psi_\alpha(z)$ when $z \in B_\alpha$.
Then $\Psi$ is a well-defined diffeomorphism. 
\end{itemize}
\begin{proof}[Proof of (G2) and (G3).]
Let $z \in M^\inter$. Then (G1) implies that there is $\alpha \in A$
such that $z \in B_\alpha$. 
If $z \in B_\beta$ also for $\beta \in A$, 
then $q(x) = q(x')$ where $x = \tilde \Psi_\alpha(z)$ and $x' = \tilde \Psi_\beta(z)$. Thus $\Psi$ is well-defined.

Note that the sets $\U_\alpha$ cover $\X$ since the sets $X_\alpha = \tilde \Psi_\alpha(B_\alpha)$ cover $\tilde \X$.
This implies that $\Psi$ is surjective.
Suppose that $\Psi(z) = \Psi(z')$ for some $z \in B_\alpha$ and $z' \in B_\beta$.
Then $q(x) = q(x')$ where $x = \tilde \Psi_\alpha(z)$ and $x' = \tilde \Psi_\beta(z)$. Thus $z = z'$ by the definition of $q$,
and we have shown that $\Psi$ is injective.

We define $\Psi_\alpha : B_\alpha \to \U_\alpha$ as the restiction $\Psi_\alpha  = \Psi|_{B_\alpha}$. It is clearly bijective.
Now $\Psi_\alpha = q_\alpha \circ \tilde \Psi_\alpha$ 
implies that $q_\alpha = \Psi_\alpha \circ \tilde \Psi_\alpha^{-1}$.
Hence the maps $q_\alpha$ are bijective. Moreover, 
if $\U = \U_\alpha \cap \U_\beta \ne \emptyset$ then we have on $q_\alpha^{-1}(\U)$ that 
$$q_\beta^{-1} \circ q_\alpha = \tilde \Psi_\beta \circ \Psi^{-1} \circ \Psi \circ \tilde \Psi_\alpha^{-1} = \tilde \Psi_\beta \circ \tilde \Psi_\alpha^{-1},$$
and we see that $q_\beta^{-1} \circ q_\alpha$ is smooth
on the open set $q_\alpha^{-1}(\U) = \tilde \Psi_\alpha(B_\alpha \cap B_\beta)$.
We have shown that the conditions (1) and (2) of \cite[Prop. 1.42]{ONeill1983}
hold. To finish the proof of (G2) we need only to verify the separation condition (3)
in \cite[Prop. 1.42]{ONeill1983}. 

Let $p, p' \in \X$ be distinct. Then we have $z \ne z'$
where $z = \Psi^{-1}(p)$ and $z' = \Psi^{-1}(p')$.
Let $\alpha,\beta \in A$ be as in (G1).
Then $\U_\alpha$ and $\U_\beta$ are disjoint sets
containing $p$ and $p'$ respectively,
since $\U_\alpha = \Psi(B_\alpha)$ and $\U_\beta = \Psi(B_\beta)$.
Now (G2) follows from \cite[Prop. 1.42]{ONeill1983}.

To show that $\Psi$ is smooth, it is enough to show that each
$q_\alpha^{-1} \circ \Psi \circ \tilde \Psi_\alpha^{-1}$
is smooth. But this is simply the identity map on $X_\alpha$.
\end{proof}

Let us show that the metric tensors $g_\alpha$
can be glued together:
\begin{itemize}
\item[(G4)] We have $(q_\alpha^{-1})^* g_\alpha = (\Psi^{-1})^* g$ 
on each $\U_\alpha$.
\end{itemize}
\begin{proof}
We recall that $g = \tilde \Psi_\alpha^* g_\alpha$ on $B_\alpha$.
Thus we have on $\U_\alpha$ that 
$$(\Psi^{-1})^* g = (\tilde \Psi_\alpha \circ \Psi^{-1})^* g_\alpha
= (\tilde \Psi_\alpha \circ \Psi_\alpha^{-1})^* g_\alpha
= (q_\alpha^{-1})^* g_\alpha.$$
\end{proof}

Let us now turn to gluing of the vector bundles $X_\alpha \times \C^n$. 
Denote by 
$\E^\alpha = (e_\ell^\alpha)_{\ell=1}^n$ the constant frame on $X_\alpha \times \C^n$ corresponding to the standard basis of $\C^n$.
Suppose that $\U_\alpha$ and $\U_\beta$ intersect 
for some indices $\alpha, \beta \in A$, and write 
$$
X_{\alpha \beta} = q_\alpha^{-1}(\U_\alpha \cap \U_\beta),
\quad
X_{\beta \alpha} = q_\beta^{-1}(\U_\alpha \cap \U_\beta).
$$
We define functions $h_1 = \Phi_\alpha^* \tilde h_1$
and $h_2 = \Phi_\beta^* \tilde h_2$,
where 
$$
\tilde h_1 = 1_{X_{\alpha \beta}} e_\ell^\alpha \in L^2(X_\alpha; \C^n),
\quad 
\tilde h_2 = 1_{X_{\beta \alpha}} a_\ell^\kappa e_\kappa^\beta
\in L^2(X_\beta; \C^n).
$$
Here $\ell, \kappa = 1, \dots, n$ and 
$a_\ell^\kappa \in C^\infty(X_{\beta \alpha})$. 
Analogously to the considerations
preceeding (\ref{eq_gluing_pullback}),
we can choose two sequences of sources 
$(f^i_j)_{j=1}^\infty,\, i=1,2$, such
that $(W f^i_j)_{j=1}^\infty$ converges to $h_i$,
and determine if (\ref{eq_gluing}) holds.
Suppose now that we have chosen $a_\ell^\kappa \in C^\infty(X_{\beta \alpha})$
so that (\ref{eq_gluing}) holds.
We define $U_{\beta \alpha} = (a_\ell^\kappa)_{\kappa,\ell=1}^n$ on $X_{\beta  \alpha}$. 
Moreover, we define an equivalence relation 
on $\tilde \X \times \C^n$ by
\begin{align}
\label{fiber_equivalence}
q(x) = q(x'), \quad \xi' = U_{\beta \alpha}(x') \xi,
\end{align}
where $x \in X_\alpha$, $x' \in X_{\beta}$ and $\xi, \xi' \in \C^n$.
We have:
\begin{itemize}
\item[(G5)] 
The equations (\ref{fiber_equivalence}) hold if and only if
$\tilde \Phi_\alpha^{-1}(x,\xi) = \tilde \Phi_\beta^{-1}(x',\xi')$.
\end{itemize}
\begin{proof}
Observe that $x \in X_\alpha$, $x' \in X_{\beta}$ and $q(x) = q(x')$ imply that 
$x' \in X_{\beta  \alpha}$.
Therefore, the second equation in (\ref{fiber_equivalence})
is well-defined whenever the first one holds. 

We write $B = B_\alpha \cap B_\beta$.
Let $Z \in \pi_E^{-1}(B)$ where $\pi_E : E|_{M^\inter} \to M^\inter$
is the bundle projection, and take $z = \pi_E(Z)$.
Moreover, denote by $Z_p = (Z_p^\ell)_{\ell=1}^n$
the representation of $Z$ in the frame $\tilde \Phi_p^* e_\ell^p$,
$p = \alpha, \beta$. Then, since $h_1$ and $h_2$ are smooth
in $B$ and satisfy (\ref{eq_gluing}),
Lemma \ref{lem_density} implies that
$$
Z = Z_\alpha^\ell \tilde \Phi_\alpha^* e_\ell^\alpha|_z
= Z_\alpha^\ell \tilde \Phi_\beta^* (a_\ell^\kappa e_\kappa^\beta)|_z
= Z_\alpha^\ell a_\ell^\kappa(\tilde \Psi_\beta(z)) \tilde \Phi_\beta^* e_\kappa^\beta|_z.
$$
Hence
$Z_\beta = U_{\beta \alpha}(\tilde \Psi_\beta(z)) Z_\alpha$.

Suppose that (\ref{fiber_equivalence}) holds, and define $Z = \tilde \Phi_\alpha^{-1}(x,\xi)$. Then $Z \in \pi_E^{-1}(B)$ and we have, using the above notation $z = \pi_E(Z)$ and  $Z_p = (Z_p^\ell)_{\ell=1}^n$, $p = \alpha, \beta$, that
$\tilde \Psi_\alpha(z) = x$ and $Z_\alpha = \xi$.
Moreover,
$\tilde \Phi_\beta(Z) = (\tilde \Psi_\beta(z), Z_\beta)$
where $\tilde \Psi_\beta(z) = x'$ as $q(x) = q(x')$,
and 
$$
Z_\beta = U_{\beta \alpha}(\tilde \Psi_\beta(z)) Z_\alpha = U_{\beta \alpha}(x') \xi = \xi'.
$$
On the other hand, if $Z = \tilde \Phi_\alpha^{-1}(x,\xi) = \tilde \Phi_\beta^{-1}(x',\xi')$, then $q(x) = q(x')$ and 
$$\xi' = Z_\beta = U_{\beta \alpha}(\tilde \Psi_\beta(z)) Z_\alpha =
U_{\beta \alpha}(x') \xi.$$
\end{proof}

We denote by $F$ the quotient space with respect to the equivalence 
(\ref{fiber_equivalence}) and by $Q : \tilde \X \times \C^n \to F$
the corresponding canonical map.
Moreover, we define 
\begin{align}
\label{11.2}
\pi_F : F \to \X:\,
\pi_F(Q(x,\xi)) = q(x), \quad (x,\xi) \in \tilde X \times \C^n,
\end{align}
and $Q_\alpha$ as 
the restriction of $Q$ on $X_\alpha \times \C^n$, $\alpha \in A$.
These maps define a smooth vector bundle structure:
\begin{itemize}
\item[(G6)] The map $\pi_F$ is a well-defined surjection and the maps $$Q_\alpha : X_\alpha \times \C^n \to \pi_F^{-1}(\U_\alpha)$$ are bijective.
There is a unique Hausdorff topology and a complete atlas on $F$
such that each $Q_\alpha^{-1}$ is a coordinate system.
The maps $\xi \mapsto Q_\alpha(x,\xi)$ are bijective from $\C^n$
to $\pi_F^{-1}(\{q(x)\})$ for $x \in X_\alpha$ and $\alpha \in A$,
and, if the fibers $\pi_F^{-1}(\{p\})$, $p \in \X$,
are equipped with the vector space structure that is pulled back from $\C^n$
via the inverses of these maps, then
$\pi_F : F \to \X$ is a smooth vector bundle that is trivial on each $\U_\alpha$.
\end{itemize}
Let us show (G6) simultaneously with the following:
\begin{itemize}
\item[(G7)] Let us define a map $\Phi : E|_{M^\inter} \to F$
by 
$\Phi(Z) = Q \circ \tilde \Phi_\alpha(Z)$ when $Z \in \pi_E^{-1}(B_\alpha)$.
Here $\pi_E$ is the bundle projection $E|_{M^\inter} \to M^\inter$.
Then $\Phi$ is a well-defined vector bundle isomorphism covering $\Psi$.
\end{itemize}
\begin{proof}[Proof of (G6) and (G7).]
Clearly $\pi_F$ is a well-defined surjection. A proof that $\Phi$
is a well-defined bijection is essentially identical with the above proof that
$\Psi$ is a well-defined bijection, and we omit it.

Let $\alpha \in A$, $x \in X_\alpha$, and consider the map $Q_{\alpha}^x(\xi) = Q_\alpha(x,\xi)$.
The definition of $\pi_F$ implies that $Q_{\alpha}^x : \C^n \to F^x$
where $F^x = \pi_F^{-1}(\{q(x)\})$.
Let us show that $Q_{\alpha}^x$ is surjective.
Let $\beta \in A$ and $x' \in X_\beta $ satisfy $q(x') = q(x)$ and let $\xi' \in \C^n$.
Then, if we choose $\xi = U_{\beta\alpha}(x')^{-1} \xi'$,
we have
$Q_\beta^{x'}(\xi') = Q_\alpha^x(\xi)$ due to (\ref{fiber_equivalence}).
Thus $Q_\alpha^x$ is surjective.
The surjectivity implies that 
$$
Q(X_\alpha \times \C^n) = \bigcup_{x \in X_\alpha} Q_{\alpha}^x(\C^n)
= \pi^{-1}_F(q(X_\alpha)) = \pi_F^{-1}(\U_\alpha).
$$

We write $E_\alpha = \pi_E^{-1}(B_\alpha)$, $F_\alpha = \pi_F^{-1}(\U_\alpha)$, and define $\Phi_\alpha = \Phi|_{E_\alpha}$.
The sets 
$$
\Phi(E_\alpha)
= Q(X_\alpha \times \C^n) = F_\alpha, \quad \alpha \in A,
$$ 
cover $F$, and $\Phi_\alpha : E_\alpha \to F_\alpha$ is bijective.
The factorization $\Phi_\alpha = Q_\alpha \circ \tilde \Phi_\alpha$
implies that $Q_\alpha$ is bijective, and
$Q_\beta^{-1} \circ Q_\alpha = \tilde \Phi_\beta \circ \tilde \Phi_\alpha^{-1}$
is smooth on the open set $Q_\alpha^{-1}(F_\alpha \cap F_\beta) = \tilde \Phi_\alpha(E_\alpha \cap E_\beta)$.

Let $p, p' \in F$, and define $z = \pi_E \circ \Phi^{-1}(p)$
and  $z' = \pi_E \circ \Phi^{-1}(p')$.
If $z \ne z'$ then we may choose $\alpha,\beta \in A$ as in (G1).
Then $E_\alpha$ and $E_\beta$ are disjoint,
whence $F_\alpha$ and $F_\beta$ are disjoint sets
containing $p$ and $p'$ respectively.
On the other hand, if $z = z'$ then there is $\alpha \in A$
such that $p,p' \in F_\alpha$.
Now \cite[Prop. 1.42]{ONeill1983} implies that $F$ has a unique smooth manifold structure.

To show that $\pi_F$ is smooth, it is enough to show that each
$q_\alpha^{-1} \circ \pi_F \circ Q_\alpha$
is smooth. But this is simply the map $\pi_\alpha : X_\alpha \times \C^n \to X_\alpha$, $\pi_\alpha(x,\xi) = x$.
A proof that $\Phi$
is smooth is essentially identical with the above proof that
$\Psi$ is smooth, and we omit it.

We define a vector space structure on $F^x$ by 
pulling back the addition and scalar multiplication via $(Q_\alpha^x)^{-1} : F^x \to \C^n$. That is,
$$
Q_\alpha^x(\xi) + c Q_\alpha^x(\eta)
= Q_\alpha^x(\xi + c \eta), \quad \xi, \eta \in \C^n\, c \in \C.
$$
Let us show that this does not depend on the choice of $x' \in q^{-1}(\{x\})$.
Suppose that $F^x = F^{x'}$ for some $\beta \in A$ and $x' \in X_\beta$,
and let $\xi',\eta' \in \C^n$.
We choose 
$\xi = U_{\beta\alpha}(x')^{-1} \xi'$
and $\eta = U_{\beta\alpha}(x')^{-1} \eta'$.
Then it holds that $Q_\beta^{x'}(\xi') = Q_\alpha^x(\xi)$, $Q_\beta^{x'}(\eta') = Q_\alpha^x(\eta)$ and $Q_\beta^{x'}(\xi' + c \eta') = Q_\alpha^x(\xi + c\eta)$
for all $c \in C$.

Next let us construct local trivializations for $F$.
We define $$\rho : \tilde \X \times \C^n \to \X \times \C^n$$ by 
$\rho = q \otimes id$, that is, $\rho(x,\xi) = (q(x), \xi)$, and set $\rho_\alpha = \rho \circ Q_\alpha^{-1}$.
Then $\rho_\alpha : F_\alpha \to \U_\alpha \times \C^n$ is a smooth bijection since 
$(q_\alpha^{-1} \otimes id) \circ \rho_\alpha \circ Q_\alpha$ is the identity on 
$X_\alpha \times \C^n$. 
Moreover, $\pi_F \circ \rho_\alpha^{-1}$ is the identity on $\U_\alpha$,
and, for $x \in X_\a$, the map $\xi \mapsto \rho_\alpha^{-1}(q(x), \xi)$
is $Q_\alpha^x$.
Thus the maps $\rho_\alpha^{-1}$, $\alpha \in A$, 
give local trivializations for $F$,
and $\pi_F : F \to \X$ is a smooth vector bundle.

Let us show that $\Phi$ is a vector bundle homomorphism. 
We recall that $q_\alpha = \Psi \circ \tilde \Psi_\alpha^{-1}$, 
$q_\alpha^{-1} \circ \pi_F \circ Q_\alpha = \pi_\alpha$ and $Q_\a= \Phi_\a \circ \widetilde \Phi^{-1}_\a$,
where $\pi_\alpha$ is the projection on right in (\ref{triv_on_B_alpha}).
thus, we have 
\begin{align} \label{11.3}
q_\alpha^{-1} \circ \pi_F \circ \Phi \circ \tilde \Phi_\alpha^{-1}
= q_\alpha^{-1} \circ \pi_F \circ Q_\alpha = \pi_\alpha,
\end{align}
and, as the diagram (\ref{triv_on_B_alpha}) commutes, we have also
\begin{align} \label{11.4}
q_\alpha^{-1} \circ \Psi \circ \pi_E \circ \tilde \Phi_\alpha^{-1}
= q_\alpha^{-1} \circ \Psi \circ \tilde \Psi_\alpha^{-1} \circ \pi_\alpha
= \pi_\alpha.
\end{align}
Thus $\pi_F \circ \Phi = \Psi \circ \pi_E$.
Let $\alpha \in A$, $z \in B_\alpha$. Then $\Phi$ is linear from the fiber $\pi_E^{-1}(\{z\})$ to 
the fiber $\pi_F^{-1}(\{\Psi(z)\})$, since 
$(Q_\alpha^x)^{-1}$,  $\Theta(\xi)^x = \tilde \Phi_\alpha^{-1}(x, \xi)$ and $(Q_\alpha^x)^{-1} \circ \Phi \circ \Theta = id$
are linear where
 $x = \tilde \Psi_\alpha(z)$ and the last equation follows from (\ref{11.3}) and (\ref{11.4}).
Hence $\Phi$ is a vector bundle homomorphism. As it is bijective, it is a vector bundle isomorphism.
\end{proof}

The connections $\nabla_\alpha$, potentials $V_\a$
and the Hermitian structures 
can be glued together:
\begin{itemize}
\item[(G8)] 
On each $\pi_F^{-1}(\U_\alpha)$,
$(Q_\alpha^{-1})^* \nabla_\alpha = (\Phi^{-1})^* \nabla$, $(Q_\alpha^{-1})^* V_\alpha = (\Phi^{-1})^* V$
and $(Q_\alpha^{-1})^* \pair{\cdot, \cdot}_{\C^n} = (\Phi^{-1})^* \pair{\cdot, \cdot}_{E}$.
\end{itemize}
A proof is essentially identical with the proof of (G4) and we omit it.

To summarize, we have shown that the following diagram 
\begin{align*}
\xymatrix{
        E \ar[r]^{\Phi} \ar[d] & F \ar[d] \\
        M^\inter \ar[r]_{\Psi}       & \X }
\end{align*}  
gives an isomorphism of the structure $(g, E, \nabla,V)$ on $M^\inter$ when $\X$ is equipped with the metric tensor given by the gluing (G4) and $F$ is equipped with the connection, the potential
and the Hermitian structure given by the gluing (G8). 
This concludes the proof of Theorem \ref{thm:maint_interior}.

Let us show that $\Phi$ extends to the accessible part $\S$ of the boundary.
If $\alpha \in A_\Gamma$, $\Gamma \in \mathcal G_\S$, corresponds to an iteration that is terminated 
immediately after the initial step, 
then we can use 
$B_\alpha = C_0$, where $C_0$ is of the form (\ref{11.5})
and  
$\tilde \Phi_\alpha = \tilde \Phi_\Gamma|_{C_0}$.
Thus $Q_\alpha^{-1} \circ \Phi|_{B_\a} = \tilde \Phi_\Gamma|_{C_0}$ extends to 
$C_0 \cup B_\p(y_0, \e_0)$ and  
\begin{align}
\label{Phi_extension_to_Gamma}
Q_\alpha^{-1} \circ \Phi = \phi_\Gamma,
\quad \text{on $B_\p(y_0, \e_0)$}.
\end{align}

\subsection{Extension to the inaccessible part of boundary}

We will give a non-constructive proof that the structure
$(g,E,\nabla,V)$ is determined up to the boundary, and this will conclude the proof of Theorem 
\ref{thm:maint}.
To this end, let 
$(M_i, g_i, E_i, \nabla_i,V_i)$, $i = 1,2$,
be two structures as in Theorem 
\ref{thm:maint}.
Let $\S_i \subset \p M_i$ be open and nonempty, 
and suppose that there is an isomorphism between the induced Hermitian vector bundles
on $\S_i$, $i =1,2$,
$$
\xymatrix{
        E_1 \ar[r]^\phi \ar[d] & E_2 \ar[d] \\
        \S_1 \ar[r]_{\psi}       & \S_2 }
$$
Note that we do not assume {\it a priori}  that $\psi$ is an isometry.

Let us choose an open cover $\mathcal G_{\S_1}$ of $\S_1$ as in the proof of Theorem \ref{thm:maint_interior}. Then for each $\Gamma_1 \in \mathcal G_{\S_1}$ there is a unitary trivialization 
\begin{align*}
\xymatrix{
        E_1 \ar[r]^{\phi_{\Gamma_1}} \ar[d] & Y_{\Gamma_1} \times \C^n \ar[d] \\
        \Gamma_1 \ar[r]_{\psi_{\Gamma_1}}       & Y_{\Gamma_1} }
\end{align*}  
We define $\Gamma_2 = \psi(\Gamma_1)$ and $\phi_{\Gamma_2} = \phi_{\Gamma_1} \circ \phi^{-1}$.
Then $\phi_{\Gamma_2} : E_2|_{\Gamma_2} \to Y_{\Gamma_1} \times \C^n$ is a unitary trivialization,
and, if $\phi$ intertwines the maps $\Lambda_{\S_1}^{2T}$ and $\Lambda_{\S_2}^{2T}$, then their representations on the respective trivializations coincide.

Theorem \ref{thm:maint_interior} implies that there is 
a Hermitian vector bundle $F \to \X$, that is equipped with a Hermitian connection $\tilde \nabla$ and a potential $\tilde V$, and whose base manifold $\X$ is equipped with a Riemannian metric $\tilde g$, such that,
for both $i=1,2$, there is an Hermitian vector bundle isomorphism
$\Phi_i : E_i|_{M_i^\inter} \to F$,
covering an isometry $\Psi_i : M_i^\inter \to \X$,
such that $\nabla_i = \Phi_i^* \tilde \nabla$ 
and $V_i = \Phi_i^* \tilde V$.
Hence $\Phi_2^{-1} \circ \Phi_1$ gives an isomorphims between the 
structures $(g_i, E_i, \nabla_i,V_i)$ on $M_i^\inter$, $i=1,2$.

It follows from \cite{Palais1957} that $\Psi = \Psi_2^{-1} \circ \Psi_1$ extends smoothly to the boundary $\p M_1$
and $(M_i, g_i)$, $i=1,2$, are isometric via the extended $\Psi$.
By considering the pullback bundle $\Psi^* E_2$, we can assume without loss of generality that $M_1 = M_2$.
Thus the following proposition implies that also 
the bundle isomorphism $\Phi = \Phi_2^{-1} \circ \Phi_1$ extends smoothly to the boundary.

\begin{Proposition} 
Let $E_i \to M$, $i=1,2$, be two Hermitian vector bundles over a smooth manifold with boundary $\p M$, and let $\nabla_i$
be a Hermitian connection on $E_i$, $i=1,2$.
Suppose that the exists a Hermitian vector bundle isomorphism $\Phi$
between $E_{1}|_{M^\inter}$ and $E_{2}|_{M^\inter}$ such that it 
covers the identity and 
that $\Phi^*\nabla_{2}=\nabla_{1}$ on $M^\inter$.
 Then $\Phi$ extends smoothly to $\partial E_1$ and the bundles and connections are isomorphic on $M$ via the extended $\Phi$.
\end{Proposition}
\begin{proof} Fix a point $x\in \partial M$ and introduce coordinates 
$$(x^{1},\dots,x^m)\in W:=[0,\varepsilon)\times (-\varepsilon,\varepsilon)^{m-1}$$ around $x$ such that
the boundary of $M$ is given by $x^{1}=0$. Without loss of generality we may assume that the bundles $E_1$ and $E_2$ are trivial over these coordinates and that $\nabla_{1}=d+A$, $\nabla_{2}=d+B$. The bundle isomorphism $\Phi$ can be represented by a smooth $U(n)$-valued function $u(x^{1},\dots,x^{m})$ defined for $x^{1}>0$ and such
that
\[B=u^{-1}du+u^{-1}Au.\]
Consider the smooth map $u_{A}:W\to U(n)$ uniquely defined by solving the following parallel transport equation along the curves $x^{1}\mapsto (x^{1},\dots,x^{m})$:
\begin{align*}
&\frac{du_{A}}{dx^{1}}+A_{(x^{1},\dots,x^{m})}(
\partial_{x^{1}})u_{A}=0,\\
&u_{A}(0,x^{2},\dots,x^{m})=Id.\\
\end{align*}
Consider a similar map $u_{B}:W\to U(n)$ associated to $B$. These two maps are convenient because, if we set
\[
\tilde{A}=u_{A}^{-1}du_{A}+u_{A}^{-1}Au_{A},\quad \tilde{B}=u_{B}^{-1}du_{B}+u_{B}^{-1}Bu_{B},
\]
then $\tilde{A}(\partial_{x^{1}})=\tilde{B}(\partial_{x^{1}})=0$. For $x^{1}>0$ define
$v=u_{A}^{-1}u u_{B}.$
Then, a simple calculation shows that 
\[\tilde{B}=v^{-1}dv+v^{-1}\tilde{A}v, \quad x^{1}>0.\]
This implies $dv(\partial_{x^{1}})=0$ and the map $v$ is independent of $x^{1}$.
Hence $v$ smoothly extends to $x^{1}=0$ and, since $u=u_{A}vu_{B}^{-1}$, $u$ is also smooth up to the boundary $x_1=0$.
\end{proof}

In order to finish the proof of Theorem \ref{thm:maint}
we still need to show that $\Phi|_{\S_1} = \phi$. 
Using the coordinate systems $Q_\alpha^{-1}$ on $F$
corresponding to choices $\alpha$ as in (\ref{Phi_extension_to_Gamma}), 
we see that $\Phi_i = \phi_{\Gamma_i}$ on $\Gamma_i$.
Thus 
$$
\Phi = \Phi_2^{-1} \circ \Phi_1 = \phi_{\Gamma_2}^{-1} \circ \phi_{\Gamma_1} = \phi
$$
on each $\Gamma_i \in \mathcal G_{\S_1}$.
This concludes the proof of Theorem \ref{thm:maint}.

\section{Calder\'on problem for connections on a cylinder} 

The proof of Corollary \ref{thm:ell} is based on a simple relation between the Dirichlet-to-Neumann map $\Lambda(\lambda)$ of the operator $-\partial_{t}^{2}+P_0 - \lambda$ and 
that of the transversal operator $P_{0}$ defined analogously to $\Lambda(\lambda)$.
That is, if $\lambda\in \mathbb{C}\setminus[\lambda_{1},\infty)$ then we define 
$$
\Lambda_0(\lambda) h = (\nabla_{0})_{\nu} u|_{\p M_0}, \quad h \in C^{\infty}(\partial M_0;E_0),
$$
where $u$ is the solution of the equation
\[(P_0-\lambda)u=0\;\text{in}\;C,\;\;\;\;\;u|_{\partial C}=f.\]

We consider an $L^2$-space with a weight in the Euclidean direction,
$$
L_\delta^2(C;E) = \{f \in L_{loc}^2(C;E);\ (1+t^2)^{\delta/2} f \in L^2(C;E) \}, \quad \delta \in \R,
$$
and define the corresponding Sobolev spaces $H^s_\delta$
analogously to \cite[Section 5]{Ferreira2013}.
Now we can formulate a relation between $\Lambda(\lambda)$
and $\Lambda_0(\lambda)$.

\begin{Proposition} Let $\lambda\in \mathbb{C}\setminus[\lambda_{1},\infty)$ and $\delta \in \R$.
Then $\Lambda(\lambda)$  extends as a bounded linear map
$\Lambda(\lambda):H_{\delta}^{3/2}(\partial C;E)\to H^{1/2}_{\delta}(\partial C;E)$.
Moreover, if $k\in\mathbb{R}$, then 
\[\Lambda_{0}(\lambda-k^2)h=e^{-kit}\Lambda(\lambda)(e^{ikt}h).\]
Note that if $h \in H^{3/2}(\p M_0; E_0)$, then $e^{ikt} h \in H_\delta^{3/2}(\p C;E)$
for any $\delta < -1/2$.
\label{prop:easy}
\end{Proposition}
\begin{proof} 
The proof that $\Lambda(\lambda)$ extends as claimed
is analogous to the scalar case \cite[Proposition 5.1]{Ferreira2013}
and we omit it.
Let $h\in H^{3/2}(\partial M_{0};E_{0})$ and let $v_h\in H^{2}(M_{0};E_{0})$ solve
\[(P_{0}-(\lambda-k^2))v_h=0\;\text{in}\;M_{0},\;\;\;\;\;v_h|_{\partial M_{0}}=h.\]
Since $\lambda\notin [\lambda_{1},\infty)$, the number $\lambda-k^2$ is not a Dirichlet eigenvalue
of $P_{0}$ and there is a unique solution $v_h$. Set
$f(t,x)=e^{ikt}h(x)$ and $u(t,x)=e^{ikt}v_{h}(x)$.
The function $u$
is in $H_{\delta}^{2}(C;E)$ for any $\delta < -1/2$, and solves
\[(-\partial_{t}^{2}+P_0-\lambda)u=0\;\text{in}\;C,\;\;\;\;\;u|_{\partial C}=f.\]
Note that $-\p_t^2 + P_0 = \nabla^* \nabla$, where $\nabla = \pi^* \nabla_0$
and  $\pi : C \to M_0$ is the canonical projection.
It follows that
\[\Lambda(\lambda)f=\nabla_{\nu}u|_{\partial C}=e^{ikt}(\nabla_{0})_{\nu}v_{h}|_{\partial M_{0}}=e^{ikt}\Lambda_0(\lambda-k^{2})h,\]
and the proposition is proved.
\end{proof}

\begin{proof}[Proof of Corollary \ref{thm:ell}]
Using that $C_{0}^{\infty}(\partial C;E)$ is dense
in $H_{\delta}^{3/2}(\partial C;E)$ for all $\delta$
together with Proposition \ref{prop:easy}, 
we can determine the map
\[\Lambda_0(\lambda-k^2):H^{3/2}(\partial M_{0};E_{0})\to H^{1/2}(\partial M_{0};E_{0})\]
for all $k\in \mathbb{R}$.
Since $\mu\mapsto \Lambda_0(\mu)$ is a meromorphic map whose poles are contanied in $\{\lambda_{1},\lambda_{2}, \dots\}$, see e.g. \cite[Lemma 4.5]{Katchalov2001}, we can recover $\Lambda_0(\mu)$ for all $\mu\in\mathbb{C}$. This is equivalent to knowing the Dirichlet-to-Neumann map $\Lambda_{\p M_0}^T$
for the wave operator $\partial_{t}^{2}+P_{0}$ for any $T>0$ \cite[Chapter 4]{Katchalov2001}. Thus Theorem \ref{thm:maint} implies
that we can recover the structure $(M_{0},g_{0},E_{0},\nabla_{0})$ as claimed.
\end{proof}

\medskip

\noindent{\bf Acknowledgements.} 
The authors thank Sergei Ivanov, Matti Lassas, Jason Lotay and Lars Louder
for helpful discussions. The authors also express their gratitude to the Institut Henri Poincar\'e and the 
organizers of the program on Inverse problems in 2015 for providing an excellent research environment while 
part of this work was in progress.  YK was  partially supported by the EPSRC grant EP/L01937X/1
and CNRS,  LO by the EPSRC grant EP/L026473/1 and 
Fondation Sciences Math\'ema\-tiques de Paris, and GPP by the EPSRC grant EP/M023842/1 and CNRS.

\bibliographystyle{abbrv}
\bibliography{main}
\end{document}

%% file: uniq_cont.pdf_tex
\begingroup%
  \makeatletter%
  \providecommand\color[2][]{%
    \errmessage{(Inkscape) Color is used for the text in Inkscape, but the package 'color.sty' is not loaded}%
    \renewcommand\color[2][]{}%
  }%
  \providecommand\transparent[1]{%
    \errmessage{(Inkscape) Transparency is used (non-zero) for the text in Inkscape, but the package 'transparent.sty' is not loaded}%
    \renewcommand\transparent[1]{}%
  }%
  \providecommand\rotatebox[2]{#2}%
  \ifx\svgwidth\undefined%
    \setlength{\unitlength}{525.57089844bp}%
    \ifx\svgscale\undefined%
      \relax%
    \else%
      \setlength{\unitlength}{\unitlength * \real{\svgscale}}%
    \fi%
  \else%
    \setlength{\unitlength}{\svgwidth}%
  \fi%
  \global\let\svgwidth\undefined%
  \global\let\svgscale\undefined%
  \makeatother%
  \begin{picture}(1,1.05443272)%
    \put(0,0){\includegraphics[width=\unitlength]{uniq_cont.pdf}}%
    \put(0.1836931,1.0081735){\color[rgb]{0,0,0}\makebox(0,0)[lb]{\smash{$t$}}}%
    \put(0.82967926,0.0089486){\color[rgb]{0,0,0}\makebox(0,0)[lb]{\smash{$\nu$}}}%
    \put(0.14411709,0.91684424){\color[rgb]{0,0,0}\makebox(0,0)[lb]{\smash{$2T$}}}%
    \put(-0.00505403,0.50281827){\color[rgb]{0,0,0}\makebox(0,0)[lb]{\smash{$2T-t_0$}}}%
    \put(0.59092703,0.0089486){\color[rgb]{0,0,0}\makebox(0,0)[lb]{\smash{$x$}}}%
  \end{picture}%
\endgroup%